\documentclass[12pt,oneside]{amsart}
\usepackage{faktor}
\usepackage[hypertexnames=false]{hyperref}
\hypersetup{
	colorlinks=true,
	linktoc=all,
	linkcolor=black,
}
\usepackage{float}
\usepackage{adjustbox}
\usepackage{graphicx}
\usepackage{latexsym,amssymb,amscd}
\usepackage{amsmath}
\usepackage[all]{xy}
\usepackage{amsthm}
\usepackage{mathrsfs}
\usepackage{times,enumerate}
\usepackage[usenames,dvipsnames]{color}
\usepackage{blkarray}
\usepackage{amsmath}
\usepackage{faktor}
\usepackage{tikz}
\usetikzlibrary{arrows,positioning}

\usetikzlibrary{arrows,decorations.pathmorphing,backgrounds,positioning,fit,matrix}
\usepackage{tikz-cd}
\usepackage{float}
\usepackage{hhline}
\usepackage[english]{babel}

\makeatletter
\DeclareRobustCommand{\loplus}{\mathbin{\mathpalette\dog@lsemi{+}}}
\DeclareRobustCommand{\lotimes}{\mathbin{\mathpalette\dog@lsemi{\times}}}
\DeclareRobustCommand{\roplus}{\mathbin{\mathpalette\dog@rsemi{+}}}
\DeclareRobustCommand{\rotimes}{\mathbin{\mathpalette\dog@rsemi{\times}}}

\newcommand{\dog@rsemi}[2]{\dog@semi{#1}{#2}{-90,90}}
\newcommand{\dog@lsemi}[2]{\dog@semi{#1}{#2}{270,90}}
\newcommand{\dog@semi}[3]{%
	\begingroup
	\sbox\z@{$\m@th#1#2$}%
	\setlength{\unitlength}{\dimexpr\ht\z@+\dp\z@\relax}%
	\makebox[\wd\z@]{\raisebox{-\dp\z@}{%
			\begin{picture}(1,1)
				\linethickness{\variable@rule{#1}}
				\roundcap
				\put(0.5,0.5){\makebox(0,0){\raisebox{\dp\z@}{$\m@th#1#2$}}}
				\put(0.5,0.5){\arc[#3]{0.5}}
			\end{picture}%
	}}%
	\endgroup
}
\newcommand{\variable@rule}[1]{%
	\fontdimen8  
	\ifx#1\displaystyle\textfont3\else
	\ifx#1\textstyle\textfont3\else
	\ifx#1\scriptstyle\scriptfont3\else
	\scriptscriptfont3\relax
	\fi\fi\fi
}
\makeatother

\usepackage{pict2e}

\usepackage{mathtools}
\usepackage{stackengine}
\setstackEOL{\\}
\usepackage{tikz-cd}
\usetikzlibrary{arrows,decorations.pathmorphing,backgrounds,positioning,fit,petri,calc,shapes.misc,decorations.markings}
\tikzset{degil/.style={
		decoration={markings,
			mark= at position 0.5 with {
				\node[transform shape] (tempnode) {$\backslash$};
			}
		},
		postaction={decorate}
	}
}

\usepackage{blkarray}

\usepackage{comment}
\usepackage{faktor}
\usepackage[hypertexnames=false]{hyperref}
\hypersetup{
	colorlinks=true,
	linktoc=all,
	linkcolor=black,
}

      \usepackage{amssymb}
      
      \usepackage{lipsum}

      \theoremstyle{plain}
      \newtheorem{theorem}{Theorem}[section]
      \newtheorem{lemma}[theorem]{Lemma}
      \newtheorem{corollary}[theorem]{Corollary}
      \newtheorem{proposition}[theorem]{Propositition}

      \theoremstyle{definition}
      \newtheorem{definition}[theorem]{Definition}
      
      \theoremstyle{remark}
      \newtheorem{remark}[theorem]{Remark}

      \makeatletter
      \def\@setcopyright{}
      \def\serieslogo@{}
      \makeatother

      \let\OLDthebibliography\thebibliography
      \renewcommand\thebibliography[1]{
      	\OLDthebibliography{#1}
      	\setlength{\parskip}{8pt}
      	\setlength{\itemsep}{0.9pt plus 0.3ex}
      }
   
\makeatletter
\g@addto@macro{\endabstract}{\@setabstract}
\newcommand{\authorfootnotes}{\renewcommand\thefootnote{\@fnsymbol\c@footnote}}%
\makeatother

\begin{document}
	
	 \author{Spyridon Afentoulidis-Almpanis}

	  \address{Dep. of Mathematics,
	  	Bar-Ilan University,
	  	Ramat-Gan, 5290002 Israel}
	
	\email{spyridon.almpanis@biu.ac.il}

	
	\title[basic simple $\mathbb{Z}_2\times \mathbb{Z}_2$-graded Lie algebras]{Structure and Representation Theory of basic simple $\mathbb{Z}_2\times \mathbb{Z}_2$-graded color Lie algebras}
	
		\maketitle
	
	\begin{abstract} We adapt methods from the theory of complex semisimple Lie algebras to develop a root theory for a class of simple $\mathbb{Z}_2 \times \mathbb{Z}_2$-graded (color) Lie algebras, which we call basic. As an application, assuming that the Cartan subalgebra is self-centralizing, we classify all finite-dimensional representations of these algebras by proving a highest weight theorem and a complete reducibility theorem.
	\end{abstract}

\tableofcontents

	\vspace{5mm}
	
	\vspace{5mm}
	
	\vspace{5mm}

	
	
	

	

	\section{Introduction}
	
	The notion of color Lie algebras was originally introduced and studied by Ree \cite{ree}, Rittenberg and Wyler \cite{ritten,ritten2}, and Scheunert \cite{scheunert} as a generalization of Lie algebras and Lie superalgebras. A color Lie algebra is a (nonassociative) graded algebra where the grading is taken with respect to some abelian group $G$ while the "symmetry" of the corresponding operation is determined by a bicharacter
	\begin{equation*}
		\epsilon: G\times G\rightarrow \mathbb{C}^\times,
	\end{equation*}
called the commutation factor on $G$.
	
	A special type of color Lie algebras are the $\mathbb{Z}_2\times\mathbb{Z}_2$-graded (color) Lie algebras. In this case, the abelian group $G$ is  $\mathbb{Z}_2\times\mathbb{Z}_2$ while the commutation factor $\epsilon$ is given by some determinant formula (for more details, see Subsection \ref{determinant}). In an analogous way, the notion of a $\mathbb{Z}_2\times\mathbb{Z}_2$-graded (color) Lie superalgebra can be defined (e.g., \cite{stoil24}).
	%
%
	These $\mathbb{Z}_2\times \mathbb{Z}_2$-graded Lie (super)algebras appear in various applications in mathematical physics. To mention some of them, they encode dynamical symmetries of L\' evy-Leblond equation, i.e., the non-relativistic limit of the Dirac equation \cite{aizawa,aisawa2,ryan25}, they have been utilized in graded quantum mechanics \cite{bruce} and the construction of classical linear and nonlinear sigma models \cite{bruce2020}, while $\mathbb{Z}_2\times \mathbb{Z}_2$-graded Lie (super)algebras appear in parastatistics \cite{tolstoy,stoil25}.
	Recently, a connection between $\mathbb{Z}_2\times\mathbb{Z}_2$-graded Lie algebras and knot invariants of Vassiliev and Kontsevich was drawn in
	\cite{aizawa2024universalweightsystemsminimal} while in \cite{leo} it is shown that $\mathbb{Z}_2\times \mathbb{Z}_2$-graded Lie algebras appear in the study of odd Khovanov homology. In
	\cite{meyer1,meyer2} a theory of algebraic Dirac operators for $\mathbb{Z}_2\times \mathbb{Z}_2$-graded Lie algebras (and more generally for color Lie algebras) was developed. We would like to draw the reader's attention to the fact that the $\mathbb{Z}_2\times\mathbb{Z}_2$-graded Lie algebras are not Lie algebras equipped with a grading in the sense of \cite{elduque}, but rather color Lie algebras.
	
In what follows, we study simple $\mathbb{Z}_2 \times \mathbb{Z}_2$-graded complex Lie algebras $\mathfrak{g}$ satisfying the condition that their Killing form is nondegenerate and that the $(0,0)$-degree component $\mathfrak{g}^{(0,0)}$ is reductive. We call such a $\mathbb{Z}_2 \times \mathbb{Z}_2$-graded Lie algebra basic.

Using methods inspired by the theory of complex semisimple Lie algebras, we develop in Section~\ref{SEC2} a root theory for these $\mathbb{Z}_2 \times \mathbb{Z}_2$-graded Lie algebras. More precisely, to every basic simple $\mathbb{Z}_2 \times \mathbb{Z}_2$-graded Lie algebra $\mathfrak{g}$, we associate an abstract root system \cite[Section II.5]{knappgreen}. As a consequence, the full machinery of abstract root systems becomes available (Weyl groups, positive root systems, simple roots, etc.).

We then use these results to classify the finite-dimensional representations of $\mathfrak{g}$. In Subsection~\ref{hwt}, we establish a highest weight theorem for finite-dimensional irreducible representations of $\mathfrak{g}$. In Subsection~\ref{weylscd}, we prove that every finite-dimensional representation of $\mathfrak{g}$ decomposes as a direct sum of irreducible subrepresentations.

Finally, we illustrate our results with two examples and conclude by raising an open question.
	
		\thanks{\textit{Acknowledgement:} The author was supported by
		the Israel Science Foundation (grant No. 1040/22). The author acknowledges support of the Institut Henri Poincar\'e (UAR 839 CNRS-Sorbonne Universit\' e), and LabEx CARMIN (ANR-10-LABX-59-01).
		
		The author would like to thank Professors C. Daskaloyannis, N. I. Stoilova, and J. Van der Jeugt for their valuable comments and suggestions.}

	\section{Basic $\mathbb{Z}_2\times\mathbb{Z}_2$-graded Lie algebras}\label{SEC2}
	In this section, we recall the notion of a $\mathbb{Z}_2\times \mathbb{Z}_2$-graded Lie algebra and we develop a root theory for a certain family of such algebras that we call basic.
	
	\subsection{$\mathbb{Z}_2\times \mathbb{Z}_2$-graded Lie algebras}\label{determinant}
	Consider the direct product $\mathbb{Z}_2\times \mathbb{Z}_2$ equipped with the standard additive operation
	\begin{equation*}
		(a_1,a_2)+(b_1,b_2):=(a_1+b_1,a_2+b_2)
	\end{equation*}
and set
\begin{equation}\label{multiplicative}
	(a_1,a_2)(b_1,b_2):=\det \begin{pmatrix} a_1&a_2\\
		b_1&b_2
		\end{pmatrix}=a_1b_2-a_2b_1\in\mathbb{Z}_2
\end{equation}
for every $(a_1,a_2),(b_1,b_2)\in\mathbb{Z}_2\times \mathbb{Z}_2$.
Observe that \eqref{multiplicative} is symmetric in the sense that 
\begin{equation*}
	(a_1,a_2)(b_1,b_2)=(b_1,b_2)(a_1,a_2).
\end{equation*}

	\begin{definition}\label{firstdef}	A \textit{$\mathbb{Z}_2\times \mathbb{Z}_2$-graded complex Lie algebra} is a $\mathbb{Z}_2\times \mathbb{Z}_2$-graded complex vector space $\mathfrak{g}=\mathfrak{g}^{(0,0)}\oplus \mathfrak{g}^{(0,1)}\oplus \mathfrak{g}^{(1,0)}\oplus \mathfrak{g}^{(1,1)}$ equipped with a bilinear map
			\begin{equation*}
				[\cdot,\cdot]:\mathfrak{g}\times\mathfrak{g}\rightarrow\mathfrak{g}
			\end{equation*} 
		satisfying
		\begin{itemize}
			\item[(i)] $[\mathfrak{g}^a,\mathfrak{g}^b]\subseteq \mathfrak{g}^{a+b}$ for every $a,b\in\mathbb{Z}_2\times \mathbb{Z}_2$,
			\item[(ii)] the \textit{$\mathbb{Z}_2\times \mathbb{Z}_2$-graded antisymmetry property}
			\begin{equation*}[x,y]=-(-1)^{ab}[y,x]
			\end{equation*}
		 for every $x\in\mathfrak{g}^a$ and $y\in\mathfrak{g}^b$,
			\item[(iii)] the \textit{$\mathbb{Z}_2\times \mathbb{Z}_2$-graded Jacobi identity}
			 \begin{equation*}
				[x,[y,z]]=[[x,y],z]+(-1)^{ab}[y,[x,z]]
					\end{equation*}
				for every $x\in\mathfrak{g}^a$, $y\in\mathfrak{g}^b$ and $z\in\mathfrak{g}$.
		\end{itemize}
	\end{definition}

\begin{remark}
	In this text, all the algebras that we consider will be complex.
\end{remark}

Notice that every Lie algebra is, in a trivial way, a $\mathbb{Z}_2\times \mathbb{Z}_2$-graded Lie algebra with $\mathfrak{g}^{(0,1)}=\mathfrak{g}^{(1,0)}=\mathfrak{g}^{(1,1)}=\{0\}$.
	On the other hand, for every $\mathbb{Z}_2\times \mathbb{Z}_2$-graded Lie algebra, the defining conditions (ii) and (iii) of Definition \ref{firstdef} for elements of the subspace $\mathfrak{g}^{(0,0)}$ become the standard antisymmetry and Jacobi identity of Lie algebras so that $\mathfrak{g}^{(0,0)}$ possesses a canonical structure of a Lie algebra.
	Moreover, for every $a\in\mathbb{Z}_2\times \mathbb{Z}_2$, the subspace $\mathfrak{g}^a$ of $\mathfrak{g}$ is a $\mathfrak{g}^{(0,0)}$-representation. 

We have a method to construct $\mathbb{Z}_2\times \mathbb{Z}_2$-graded Lie algebras starting from $\mathbb{Z}_2\times \mathbb{Z}_2$-graded algebras. More precisely, if $A$ is an associative $\mathbb{Z}_2\times \mathbb{Z}_2$-graded algebra, then $A$ equipped with the bracket 
	\begin{equation}\label{assocbracket}
		[x,y]:=xy-(-1)^{ab}yx, \quad x\in\mathfrak{g}^a, y\in\mathfrak{g}^b
	\end{equation}
is a $\mathbb{Z}_2\times \mathbb{Z}_2$-graded Lie algebra. 
An important example of such a $\mathbb{Z}_2\times \mathbb{Z}_2$-graded algebra arises as the algebra of $\mathbb{Z}_2\times \mathbb{Z}_2$-graded endomorphisms of $\mathbb{Z}_2\times \mathbb{Z}_2$-graded vector spaces. More precisely, if $V$ is a
$\mathbb{Z}_2\times \mathbb{Z}_2$-graded finite-dimensional vector space, for every $a\in \mathbb{Z}_2\times\mathbb{Z}_2$, set
	\begin{equation*}
		\mathrm{End}^a(V):=\{ T\in\mathrm{End}(V)\mid T(V^b)\subset V^{a+b}, \forall b\in\mathbb{Z}_2\times\mathbb{Z}_2\}.
	\end{equation*}
Then, the space 
\begin{equation*}\mathrm{End}(V):=\bigoplus\limits_{a\in \mathbb{Z}_2\times\mathbb{Z}_2}	\mathrm{End}^a(V)
\end{equation*}
is a $\mathbb{Z}_2\times\mathbb{Z}_2$-graded algebra. Equipped with the bracket \eqref{assocbracket}, $\mathrm{End}(V)$ is a $\mathbb{Z}_2\times\mathbb{Z}_2$-graded Lie algebra, denoted by $\mathfrak{gl}(V)$, that we call \textit{the general linear $\mathbb{Z}_2\times\mathbb{Z}_2$-graded Lie algebra of $V$}. 
 In \cite{ritten2,stoil}, the authors provide families of $\mathbb{Z}_2\times\mathbb{Z}_2$-graded Lie algebras which, in a certain sense, are analogous to the classical Lie algebras. 
 
 \subsection{Basic simple $\mathbb{Z}_2\times \mathbb{Z}_2$-graded Lie algebras}

Recall that a $\mathbb{Z}_2\times\mathbb{Z}_2$-graded Lie algebra $\mathfrak{g}$ is said to be \textit{simple} whenever it does not contain 
 any nontrivial proper $\mathbb{Z}_2\times\mathbb{Z}_2$-graded ideal. In what follows, we will be interested in a certain family of simple $\mathbb{Z}_2\times\mathbb{Z}_2$-graded Lie algebras which we will call basic.
 
 Let $\mathfrak{g}$ be a $\mathbb{Z}_2\times\mathbb{Z}_2$-graded Lie algebra $\mathfrak{g}$ and let $\mathrm{ad}$ stand for \textit{the adjoint map of $\mathfrak{g}$}, i.e., the $\mathbb{Z}_2\times\mathbb{Z}_2$-graded derivation 
 \begin{equation*}
 	\mathrm{ad}:	\mathfrak{g}\rightarrow \mathfrak{gl}(\mathfrak{g})
 \end{equation*}
 of $\mathfrak{g}$
 defined by 
 $	\mathrm{ad}(x)(y):=[x,y]$, $ x,y\in\mathfrak{g}$. We define \textit{the Killing form of $\mathfrak{g}$} to be the bilinear form
 \begin{equation*}
 	K(x,y)=\mathrm{tr}(\mathrm{ad}x\circ\mathrm{ad}y), \quad x,y\in\mathfrak{g}.
 \end{equation*}
 We note, that $K$ is symmetric while it satisfies the identity
 \begin{equation*}
 	K([x,y],z)+(-1)^{ab}K(y,[x,z])=0, \quad \forall x\in\mathfrak{g}^a,y\in \mathfrak{g}^{b}, z\in\mathfrak{g}.
 \end{equation*}
 Moreover, since for $x\in\mathfrak{g}^a$ and $y\in\mathfrak{g}^b$,
 \begin{equation*}
 	\big(\mathrm{ad}(x)\circ \mathrm{ad}(y)\big)(\mathfrak{g}^c)\subseteq \mathfrak{g}^{a+b+c},
 \end{equation*}
 we have $K(x,y)\neq0$ only if  $a=b$. \\

 \begin{definition}\label{defafterdisc} A $\mathbb{Z}_2\times\mathbb{Z}_2$-graded Lie algebra is \textit{basic} if its Killing form $K$ is nondegenerate while the $(0,0)$-degree subspace $\mathfrak{g}^{(0,0)}$ is reductive. 
 	\end{definition}
 
 \begin{remark}\label{remark} We note that the "classical"  $\mathbb{Z}_2\times\mathbb{Z}_2$-graded Lie algebras appearing in \cite{stoil,stoil25} are basic. 
 	\end{remark}
 
 %
 
 
 \begin{lemma}\label{nonde}
 	If $\mathfrak{g}$ is a simple $\mathbb{Z}_2\times\mathbb{Z}_2$-graded Lie algebra with $\mathfrak{g}^{(0,0)}$ being reductive and nonabelian, then the Killing form $K$ of $\mathfrak{g}$ is nondegenerate so that $\mathfrak{g}$ is basic.
 \end{lemma}

\begin{proof}
Since the Killing form $K$ is homogeneous of degree $(0,0)$, its radical 
	\begin{equation*}
		\mathrm{rad}(K)\coloneqq \{x\in\mathfrak{g}\mid K(x,y)=0\text{ for all } y\in\mathfrak{g}\}
	\end{equation*}
is a $\mathbb{Z}_2\times\mathbb{Z}_2$-graded ideal of $\mathfrak{g}$ and from the simplicity of $\mathfrak{g}$ we deduce that $\mathrm{rad}(K)=\{0\}$ or $\mathrm{rad}(K)=\mathfrak{g}$.

We will show that $\mathrm{rad}(K)=\{0\}$. Suppose that $\mathrm{rad}(K)=\mathfrak{g}$ so that $K$ is trivial. Since $\mathfrak{g}^{(0,0)}$ is a nonabelian reductive Lie algebra, the derived Lie algebra $[\mathfrak{g}^{(0,0)},\mathfrak{g}^{(0,0)}]$ is semisimple so that it contains an $\mathfrak{sl}(2,\mathbb{C})$-copy, say $\mathfrak{s}$, generated by a standard $\mathfrak{sl}_2$-triplet $\{h,e,f\}$. 
The subalgebra $\mathfrak{s}$ acts on $\mathfrak{g}$, and the representation theory of $\mathfrak{sl}(2,\mathbb C)$ yields an
 eigenspace decomposition 
\begin{equation*}
	\mathfrak{g}=\bigoplus\limits_{n\in\mathbb{Z}}\mathfrak{g}(n)
\end{equation*}
for the action of $h$.
Then 
\begin{equation*}
	0=K(h,h)=\mathrm{tr}(\mathrm{ad}h\circ\mathrm{ad}h)=\sum\limits_{n\in\mathbb{Z}}n^2\dim\mathfrak{g}(n)
\end{equation*}
and we deduce that $\mathfrak{g}(0)=\mathfrak{g}$. In other words, $\mathfrak{s}$ acts trivially on $\mathfrak{g}$ so that $\mathfrak{s}$ is a nonzero $\mathbb{Z}_2\times\mathbb{Z}_2$-graded proper ideal of $\mathfrak{g}$ which is absurd because $\mathfrak{g}$ is simple and $\mathfrak{s}\subseteq\mathfrak{g}^{(0,0)}\subsetneq\mathfrak{g}$. Hence $\mathrm{rad}(K)=\{0\}$ and $K$ is nondegenerate.
\end{proof}

 \hspace{-4mm}\textbf{Assumption:} For the rest of our text, we assume that $\mathfrak{g}$ is a basic simple $\mathbb{Z}_2\times\mathbb{Z}_2$-graded Lie algebra. We may assume that $\mathfrak{g}$ is not a Lie algebra; hence $\mathfrak{g}\neq \mathfrak{g}^{(0,0)}$.

\subsection{Root theory for basic simple $\mathbb{Z}_2\times\mathbb{Z}_2$-graded Lie algebras}
Let $\mathfrak{t}$ be a Cartan subalgebra of the reductive Lie algebra $\mathfrak{g}^{(0,0)}$ and for every $\alpha\in\mathfrak{t}^*$ set 
\begin{equation*}
	\mathfrak{g}_\alpha=\{x\in\mathfrak{g}\mid \big(\mathrm{ad}(h)-\alpha(h)\mathrm{Id}\big)^nx=0, \forall h\in\mathfrak{t} \text{ for some } n=n(h,x)>0\}.
\end{equation*}
By abuse of terminology, we say that $\mathfrak{t}$ is a \textit{Cartan subalgebra of $\mathfrak{g}$}. Moreover, whenever $\alpha\neq0$ and $\mathfrak{g}_\alpha\neq \{0\}$, we say that  $\alpha\in\mathfrak{t}^*$ is a \textit{generalized root} of $\mathfrak{g}$ (associated with $\mathfrak{t}$) and $\mathfrak{g}_\alpha$ is \textit{the generalized root subspace} corresponding to $\alpha$. We denote by $\Delta$ the set of all generalized roots of $\mathfrak{g}$. 

A standard linear algebra argument shows that $\mathfrak{g}$ admits a decomposition
	\begin{equation}\label{triangular}
	\mathfrak{g}=\mathfrak{t}\oplus\bigoplus\limits_{\alpha\in\Delta}\mathfrak{g}_\alpha\oplus\bigoplus\limits_{a\neq (0,0)}\mathfrak{g}_0^a
\end{equation}
where $
	\mathfrak{g}^a_0\coloneqq\mathfrak{g}_0\cap\mathfrak{g}^{a}.
$
A standard argument (e.g., \cite[Proposition 2.5]{knappgreen}), shows that for every $\alpha,\beta\in\mathfrak{t}^*$, we have
\begin{equation*}
	[\mathfrak{g}_\alpha,\mathfrak{g}_\beta]\subseteq \mathfrak{g}_{\alpha+\beta}.
\end{equation*}
Moreover, whenever $\alpha,\beta\in\mathfrak{t}^*$ with $\alpha+\beta\neq0$, the corresponding generalized root subspaces $\mathfrak{g}_\alpha$ and $\mathfrak{g}_\beta$ are orthogonal with respect to $K$. In combination with the fact that $\mathfrak{g}^{a}$ and $\mathfrak{g}^b$ are orthogonal whenever $a\neq b$, we obtain the following lemma.

\begin{lemma}\label{pairing}  Let $\alpha,\beta\in\mathfrak{t}^*$ and $a,b\in\mathbb{Z}_2\times\mathbb{Z}_2$. 
	\begin{itemize}
		\item[(i)] The subspace $\mathfrak{g}_\alpha^a\coloneqq \mathfrak{g}_\alpha\cap \mathfrak{g}^a$ is nontrivial if and only if the subspace $\mathfrak{g}_{-\alpha}^a$ is nontrivial. In this case,	
	 the Killing form $K$ defines a nondegenerate pairing between $\mathfrak{g}_\alpha^a$ and $\mathfrak{g}_{-\alpha}^a$. 
	 \item[(ii)] If $\alpha+\beta\neq 0$ or $a\neq b$, the subspaces $\mathfrak{g}_\alpha^a$ and $\mathfrak{g}_\beta^b$ are orthogonal with respect to the Killing form $K$.
	\end{itemize}
\end{lemma}

As a consequence, $\alpha$ is a generalized root if and only if $-\alpha$ is a generalized root. Morevover, the restriction of the Killing form $K$ to $\mathfrak{t}$ is nondegenerate so that for every element $\alpha\in\mathfrak{t}^*$ there is a unique element $H_\alpha\in\mathfrak{t}$ such that 
\begin{equation*}
	\alpha(H)=K(H_\alpha,H) \text{ for all }H\in\mathfrak{t}.
\end{equation*}
We will denote by $\langle\cdot,\cdot\rangle$ the bilinear symmetric form defined on $\mathfrak{t}^*\times\mathfrak{t}^*$ by
\begin{equation*}
	\langle \alpha,\beta\rangle:=K(H_\alpha,H_\beta), \quad \alpha,\beta\in\mathfrak{t}^*.
\end{equation*}

For every nonzero subspace $\mathfrak{g}_\alpha^a$, fix a nonzero element $E_\alpha^a\in\mathfrak{g}_\alpha^a$ such that 
\begin{equation}\label{bracketaki}
	[H,E_\alpha^a]=\alpha(H)E_\alpha^a
\end{equation}
for every $H\in\mathfrak{t}$. This is always possible due to Lie's Theorem \cite[Theorem 1.25]{knappgreen}.

\begin{lemma}\label{XEalpha}
	For every $X\in\mathfrak{g}_{-\alpha}^a$, we have
	\begin{equation*}
		[E_\alpha^a,X]=K(E_\alpha^a,X)H_\alpha.
	\end{equation*}
\end{lemma}

\begin{proof} The proof follows a standard argument (e.g., \cite[Lemma 2.18]{knappgreen}).
	For every $H\in\mathfrak{t}$, we have
	\begin{align*}
		K([E_\alpha^a,X],H)&=(-1)^{a(a+(0,0))}K(X,[H,E_\alpha^a])\\
		&=K(X,[H,E_\alpha^a])\\
		&=\alpha(H)K(X,E_\alpha^a)\\
		&=K(H,H_\alpha)K(X,E_\alpha^a)\\
		&=K(K(X,E_\alpha^a)H_\alpha,H)
	\end{align*}
so that 
\begin{equation*}
	[E_\alpha^a,X]=K(E_\alpha^a,X)H_\alpha+\text{terms from }\bigoplus\limits_{b\neq(0,0)}\mathfrak{g}_0^b.
	\end{equation*}
Since $[E_\alpha^a,X]\in[\mathfrak{g}^a,\mathfrak{g}^a]\subseteq \mathfrak{g}^{(0,0)}$, we conclude that 
\begin{equation*}
	[E_\alpha^a,X]=K(E_\alpha^a,X)H_\alpha.
\end{equation*}
\end{proof}

\begin{lemma}\label{onedim}
	For every $\alpha\in\mathfrak{t}^*$ and $a\in\mathbb{Z}_2\times\mathbb{Z}_2$, we have
	\begin{equation*}
		\dim\mathfrak{g}_\alpha^a\leq 1.
		\end{equation*}
\end{lemma}

\begin{proof}
	Fix $\alpha\in \Delta$ and $a\in\mathbb{Z}_2\times\mathbb{Z}_2$ and assume that $\mathfrak{g}_\alpha^a$ is nonzero so that there exists a nonzero element $E_\alpha^a\in\mathfrak{g}_\alpha^a$ satisfying \eqref{bracketaki}. According to Lemma \ref{pairing}, there exists
	 $X_{-\alpha}^a\in\mathfrak{g}_{-\alpha}^a$ such that $K(E_\alpha^a,X_{-\alpha}^a)=1$. Set $\mathfrak{s}_\alpha^a$ to be the $\mathbb{Z}_2\times\mathbb{Z}_2$-graded Lie subalgebra generated by $H_\alpha,E_\alpha^a$ and $X_{-\alpha}^a$. Although 
	 $\mathfrak{s}_\alpha^a$
	 is defined as a graded subalgebra, it is in fact an ordinary Lie algebra, since $\mathfrak{g}^{(0,0)}\oplus\mathfrak{g}^a$
	 is closed under the bracket and the graded antisymmetry and Jacobi identities reduce to the usual Lie algebra identities on this subspace.
	 Notice that since $X_{-\alpha}^a$ is not necessarily a proper root vector of $\mathfrak{t}$, the Lie algebra $\mathfrak{s}_\alpha^a$ may be of dimension greater than three. 
	 
	 Set
	\begin{equation*}
		M_1\coloneqq \mathbb{C}E_\alpha^a\oplus \mathbb{C}H_\alpha\oplus\bigoplus\limits_{n\geq 1}\mathfrak{g}_{-n\alpha}^{na}.
		\end{equation*}
	Let us check that $\mathfrak{s}_\alpha^a$ acts on $M_1$.
	It suffices to check that the generators $E_\alpha^a$, $H_\alpha$ and $X_{-\alpha}^a$ of $\mathfrak{s}_\alpha^a$ preserves $M_1$. It is clear that $H_\alpha$ acts on $M_1$. For $E_\alpha^a$, we have
	\begin{equation*}
		[E_\alpha^a,E_\alpha^a]=0,\quad [E_\alpha^a,H_\alpha]=-\alpha(H_\alpha)E_\alpha^a\in M_1
	\end{equation*}
and 
\begin{equation*}
	[E_\alpha^a,\mathfrak{g}_{-n\alpha}^{na}]\subseteq \mathfrak{g}_{-(n-1)\alpha}^{(n+1)a}=\mathfrak{g}_{-(n-1)\alpha}^{(n-1)a}\subset M_1,
\end{equation*}
for $n\geq 2$ while, according to Lemma \ref{XEalpha}, for $n=1$ we have
\begin{equation*}
	[E_\alpha^a,\mathfrak{g}_{-\alpha}^a]\subseteq \mathbb{C}H_\alpha\subseteq M_1.
\end{equation*}
For $X_{-\alpha}^a$, we have
\begin{equation*}
	[X_{-\alpha}^a,E_{\alpha}^a]=-K(E_{\alpha}^a,X_{-\alpha}^a)H_\alpha=-H_\alpha\in M_1,\quad  [X_{-\alpha}^a,H_{\alpha}]\in\mathfrak{g}_{-\alpha}^a\subset M_1
\end{equation*}
and 
\begin{equation*}
	[X_{-\alpha}^a,\mathfrak{g}_{-n\alpha}^{na}]\subseteq \mathfrak{g}_{-(n+1)\alpha}^{(n+1)a}\subset M_1.
\end{equation*}
Therefore, $M_1$ is a representation of the Lie algebra $\mathfrak{s}_\alpha^a$.

The element $H_\alpha=[E_\alpha^a,X_{-\alpha}^a]$ acts by a traceless operator on $M_1$ so that
\begin{align*}
	0&=\mathrm{tr}(\mathrm{ad}H_\alpha)=\alpha(H_\alpha)+0-\alpha(H_\alpha)\sum\limits_{n\geq 1}n\dim \mathfrak{g}_{-n\alpha}^{na}\\&=\alpha(H_\alpha)(1-\sum\limits_{n\geq0}n\dim\mathfrak{g}_{-n\alpha}^{na}).
\end{align*} 
In order to prove the statement, it suffices to show that $\alpha(H_\alpha)\neq 0$. Indeed, if $\alpha(H_\alpha)\neq0$, then 
\begin{equation}\label{dimns1}
	\sum\limits_{n\geq0}n\dim\mathfrak{g}_{-n\alpha}^{na}=1
\end{equation}
and we deduce that $\dim\mathfrak{g}_{-\alpha}^a=1$. Then, from Lemma \ref{pairing}(i), $\dim\mathfrak{g}_\alpha^a=1$. 

 Let us now prove that $\alpha(H_\alpha)\neq0$. Suppose that $\alpha(H_\alpha)=0$. Choose $\beta\in\Delta$ and set 
\begin{equation*}
	M_2\coloneqq \bigoplus\limits_{n\in\mathbb{Z}}\mathfrak{g}_{\beta+n\alpha}.
\end{equation*}
Then $M_2$ is a representation of  $\mathfrak{s}_\alpha^a$ and the element $H_\alpha$ acts, as before, by a traceless operator on $M_2$. Hence
\begin{equation*}
	0=\sum\limits_{n\in\mathbb{Z}}\big(\beta(H_\alpha)+n\alpha(H_\alpha)\big)\dim\mathfrak{g}_{\beta+n\alpha}=\beta(H_\alpha)\sum\limits_{n\in\mathbb{Z}}\dim\mathfrak{g}_{\beta+n\alpha}
\end{equation*}
and $\beta(H_\alpha)=0$. Since $\beta$ was chosen to be any element of $\Delta$, we conclude that $\mathrm{ad}H_\alpha$ acting on $\mathfrak{g}$ is a nilpotent operator so that for any $H\in\mathfrak{t}$, $\mathrm{ad}H_\alpha\circ \mathrm{ad}H$ is a nilpotent operator, too. Then
\begin{equation*}
	K(H_\alpha,H)=\mathrm{tr}(\mathrm{ad}H_\alpha\circ\mathrm{ad}H)=0
\end{equation*}
for every $H\in\mathfrak{t}$ and from Lemma \ref{nonde} we conclude that $H_\alpha=0$ which is absurd. Hence $\alpha(H_\alpha)\neq 0$.
\end{proof}

\begin{remark}
	From \eqref{dimns1}, one deduces that for every $\alpha\in\Delta$ and $n>1$ we have $\mathfrak{g}_{n\alpha}=\{0\}$.
	\end{remark}

 \begin{corollary}\label{split} For every $\alpha\in\Delta$, the generalized root subspace $\mathfrak{g}_\alpha$ is a root subspace, i.e., 
 	\begin{equation*}
 		\mathfrak{g}_\alpha=\{X\in\mathfrak{g}\mid \mathrm{ad}H(X)=\alpha(H)X\text{ for every } X\in\mathfrak{g}\}.
 	\end{equation*}
 \end{corollary}

\begin{proof}
	The subspace $\mathfrak{g}_\alpha^a$ decomposes into
	\begin{equation*}
		\mathfrak{g}_\alpha=\bigoplus\limits_{a\in\mathbb{Z}_2\times\mathbb{Z}_2}\mathfrak{g}_\alpha^a
	\end{equation*}
where each $\mathfrak{g}_\alpha^a$ is, according to Lemma \ref{onedim}, $1$-dimensional and hence generated by a proper root vector.
\end{proof}

As a consequence of Corollary \ref{split}, for every $\alpha\in\Delta$ and $a\in\mathbb{Z}_2\times\mathbb{Z}_2$, whenever $\mathfrak{g}_\alpha^a\neq 0$, $\mathfrak{g}_\alpha^a$ is generated by the element $E_\alpha^a$ defined above and satisfying \eqref{bracketaki}. Fix a root $\alpha\in\Delta$ and $a\in\mathbb{Z}_2\times\mathbb{Z}_2$ with $\mathfrak{g}_\alpha^a\neq 0$ and consider the Lie algebra $\mathfrak{s}_\alpha^a$ defined in the proof of Lemma \ref{onedim}, i.e., the Lie algebra generated, as a Lie algebra, by the elements $H_\alpha,E_\alpha^a$ and $E_{-\alpha}^a$. Then $\mathfrak{s}_\alpha^a$, now as a vector space, is generated by $H_\alpha, E_\alpha^a$ and $E_{-\alpha}^a$. Indeed, we have
	\begin{align*}
	[H_\alpha,E_\alpha^a]&=\alpha(H_\alpha)E_\alpha^a,\\
	[H_{\alpha},E_{-\alpha}^a]&=-\alpha(H_\alpha)E_{-\alpha}^a
\end{align*}
	and, according to Lemma \ref{nonde},
\begin{equation*}
	[E_\alpha^a,E_{-\alpha}^a]=K(E_\alpha^a,E_{-\alpha}^a)H_\alpha.
\end{equation*}
According to Lemma \ref{pairing}, $K(E_\alpha^a,E_{-\alpha}^a)\neq 0$ and we may assume that \begin{equation*}K(E_\alpha^a,E_{-\alpha}^a)=1.
	\end{equation*}
 On the other hand, in the proof of Lemma \ref{onedim}, we saw that \begin{equation*}\langle\alpha,\alpha\rangle=a(H_\alpha)\neq0.
 	\end{equation*}
Set
\begin{align*}
	h_\alpha&\coloneqq \frac{2}{\langle\alpha,\alpha\rangle}H_\alpha,\\
	x_\alpha^a&\coloneqq \frac{2}{\langle\alpha,\alpha\rangle}E_\alpha^a,\\
	y_\alpha^a&\coloneqq E_{-\alpha}^a.
\end{align*}
Then the elements $h_\alpha,x_\alpha^a$ and $y_\alpha^a$ form a standard $\mathfrak{sl}_2$-triplet for $\mathfrak{s}_\alpha^a$. Therefore, we obtain the following lemma which will be utilized extensively in what follows.

\begin{lemma}\label{sltriple}
	The Lie algebra $\mathfrak{s}_\alpha^a$ generated by the elements $h_\alpha, x_\alpha^a$ and $y_{\alpha}^a$ is isomorphic to $\mathfrak{sl}(2,\mathbb{C})$.
\end{lemma}

 Let 
\begin{equation*}
	\mathfrak{t}_0\coloneqq \sum\limits_{\alpha\in\Delta}\mathbb{R}H_\alpha.
\end{equation*}
By using a standard argument (e.g., \cite[Corollary 2.38]{knappgreen}), one can show that $\mathfrak{t}_0$ is a real form of $\mathfrak{t}$. Namely, if an element $H$ of $\mathfrak{t}$ is orthogonal to $\mathfrak{t}_0$, then it must be a central element of $\mathfrak{g}$ so that $H$ must be trivial. Moreover, the restriction of the complex bilinear form $\langle\cdot,\cdot\rangle$ on $\mathfrak{t}_0\times\mathfrak{t}_0$ is a positive-definite inner product. All we have to do is to reproduce arguments used in the context of complex semisimple Lie algebras \cite[Section II.4]{knappgreen}. For the sake of completeness, let us shortly present the argument which shows that $\langle\cdot,\cdot\rangle$ is an inner product on $\mathfrak{t}_0$.
 For $\beta\in\Delta$, consider the representation
	\begin{equation*}
	M_2=\bigoplus\limits_{n\in\mathbb{Z}}\mathfrak{g}_{\beta+n\alpha}
\end{equation*}
of $\mathfrak{s}_\alpha^a$ and the set
\begin{equation*}
\{n\in\mathbb{Z}\mid \beta+n\alpha\in\Delta\sqcup\{0\}\}.
\end{equation*}
Let $-p_\beta$ (respectively $q_\beta$) be the smallest (respectively largest) element of the above set.
Then from the representation theory of $\mathfrak{s}_\alpha^a\cong \mathfrak{sl}(2,\mathbb{C})$ on $M_2$, we deduce \cite[Proposition 2.29]{knappgreen} that
\begin{equation}\label{realrel}
	p_\beta-q_\beta=2\frac{\langle\beta,\alpha\rangle}{\langle\alpha,\alpha\rangle}
\end{equation}
and 
\begin{align*}
	\langle\alpha,\alpha\rangle&=K(H_\alpha,H_\alpha)\\
	&=\mathrm{tr}(\mathrm{ad}H_\alpha\circ \mathrm{ad}H_\alpha)\\
	&=\sum\limits_{\gamma\in \Delta} \gamma(H_\alpha)^2\dim\mathfrak{g}_\gamma\\
	&=\sum\limits_{\gamma\in \Delta} \langle\gamma,\alpha\rangle^2\dim\mathfrak{g}_\gamma\\
	&=\sum\limits_{\gamma\in \Delta} \dfrac{\langle\alpha,\alpha\rangle^2}{4}\big(2\frac{\langle\gamma,\alpha\rangle}{\langle\alpha,\alpha\rangle}\big)^2\dim\mathfrak{g}_\gamma\\
	&=\sum\limits_{\gamma\in \Delta} \dfrac{\langle\alpha,\alpha\rangle^2}{4}\big(p_\gamma-q_\gamma)^2\dim\mathfrak{g}_\gamma.
\end{align*}
Since from the proof of Lemma \ref{onedim} we have $\alpha(H_\alpha)=\langle\alpha,\alpha\rangle\neq0$, we deduce that
\begin{equation*}
	\langle\alpha,\alpha\rangle=\frac{4}{\sum\limits_{\gamma\in \Delta} \big(p_\gamma-q_\gamma)^2\dim\mathfrak{g}_\gamma}
\end{equation*}
so that $\langle\alpha,\alpha\rangle>0$.
This shows that $\langle\cdot,\cdot\rangle$ is a real inner product of $\mathfrak{t}_0$.
Let us note that as a direct consequence of \eqref{realrel}, we deduce that every root $\alpha\in\Delta$ is real-valued when restricted to $\mathfrak{t}_0$.

For every $\alpha\in\Delta$, let $s_\alpha$ be the reflection of $\mathfrak{t}^*$ given by
\begin{equation*}
	s_\alpha(\beta):=\beta-2\frac{\langle \beta,\alpha\rangle}{\langle \alpha,\alpha\rangle}\alpha,\quad \beta\in\mathfrak{t}^*.
\end{equation*}
We say that the group $W_\mathfrak{g}$ generated by all reflections $s_\alpha$, $\alpha\in\Delta$, is \textit{the Weyl group of $\mathfrak{g}$}. For more details concerning Weyl groups of graded Lie algebras, see \cite{vinberg77}.
One can show that the Weyl group $W_\mathfrak{g}$ permutes the elements of $\Delta$. Namely, let $\alpha,\beta\in\Delta$ and consider the representation $M_2$ of $\mathfrak{s}_\alpha^a\cong \mathfrak{sl}(2,\mathbb{C})$ and the integers $p_\beta,q_\beta$ defined above. Then, the representation theory of $\mathfrak{s}_\alpha^a\cong\mathfrak{sl}(2,\mathbb{C})$ ensures that for every $n\in{-p_\beta,\ldots,q_\beta}$ we have
$\beta+n\alpha\in\Delta$ so that 
\begin{equation*}
	s_\alpha(\beta)=\beta-2\frac{\langle\beta,\alpha\rangle}{\langle\alpha,\alpha\rangle}\alpha =\beta+(q_\beta-p_\beta)\alpha\in\Delta.
\end{equation*}

The above discussion is summarized in the following corollary.

\begin{corollary}\label{abstractrootsystem}
	The set $\Delta$ is an abstract root system \cite[Section II.5]{knappgreen} in the real inner product form $\mathfrak{t}_0^*$ of $\mathfrak{t}^*$.
\end{corollary}

Hence, we can use all the machinery (positive systems, simple roots, Weyl chambers, transitivity of Weyl group on Weyl chambers, etc.) coming from the theory of abstract root systems (e.g., \cite[Sections II.5-6]{knappgreen}). 


Let us say a few words about the case when 
 $\mathfrak{t}$ is self-centralizing in $\mathfrak{g}$, i.e., the centralizer 
\begin{equation*}
	\mathfrak{z}_\mathfrak{g}(\mathfrak{t})\coloneqq \{x\in\mathfrak{g}\mid [x,\mathfrak{t}]=\{0\}\}
\end{equation*}
of $\mathfrak{t}$ in $\mathfrak{g}$ coincides with $\mathfrak{t}$. This is the case, for instance, for the classical $\mathbb{Z}_2\times\mathbb{Z}_2$-graded Lie algebras $\mathfrak{sl}_{p,q,r,s}(n)$, $\mathfrak{sp}_p(n)$ and $\mathfrak{so}_p(n)$ of \cite{stoil,stoil25}.

\begin{lemma}\label{scentr}
	 If the Cartan subalgebra $\mathfrak{t}$ of $\mathfrak{g}$ is self-centralizing then 
	 \begin{equation*}
	 	\mathfrak{g}_0\coloneqq\{x\in\mathfrak{g}\vert (\mathrm{ad}h)^nx=0, \forall h\in\mathfrak{t} \hspace{1mm}\text{ for some } n=n(h,x)>0\}
	 	\end{equation*}
	coincides with $\mathfrak{t}$.
\end{lemma}

\begin{proof} Of course, $\mathfrak{t}\subseteq \mathfrak{g}_0$. For the other inclusion, assume that $X\in\mathfrak{g}_0$ and decompose $X$ into a linear sum of homogeneous elements
	\begin{equation*}
		X=X^{(0,0)}+X^{(0,1)}+X^{(1,0)}+X^{(1,1)}. 
	\end{equation*}
Then $X^a\in\mathfrak{g}_0$ for every $a\in\mathbb{Z}_2\times\mathbb{Z}_2$. Hence, without loss of generality, we may assume that $X$ is homogeneous, say of degree $b$. Then there is an element $X'\in\mathfrak{g}_0^b$ such that $X'\in\mathfrak{z}_\mathfrak{g}(\mathfrak{t})$. Since $\mathfrak{z}_\mathfrak{g}(\mathfrak{t})=\mathfrak{t}\subseteq\mathfrak{g}^{(0,0)}$, we deduce that $X'\in\mathfrak{g}^{(0,0)}$ and $b=(0,0)$. Therefore $X\in\mathfrak{g}_0^{(0,0)}=\mathfrak{t}$ and we conclude that $\mathfrak{g}_0\subseteq\mathfrak{t}$.
\end{proof}

\begin{remark} As a consequence of Lemma \ref{scentr}, if $\mathfrak{t}$ is self-centralizing, then $\mathfrak{g}$ is a split $\mathbb{Z}_2\times\mathbb{Z}_2$-graded Lie algebra in the sense of \cite{rootcolor}.
\end{remark}

\begin{lemma}\label{oned}
	If the Cartan subalgebra $\mathfrak{t}$ is self-centralizing, we have
	\begin{equation*}
		\mathrm{dim}\mathfrak{g}_\alpha= 1
	\end{equation*}
	for every $\alpha\in\Delta$.
	Moreover, if $\alpha\in\Delta$, then $n\alpha\in\Delta$ if and only if $n=\pm1$.
\end{lemma}

\begin{proof}
	The proof is similar to the proof of Lemma \ref{onedim} or \cite[Proposition 2.21]{knappgreen}. More precisely, the Lie algebra $\mathfrak{s}_\alpha^a\cong\mathfrak{sl}(2,\mathbb{C})$ acts on the space
	\begin{equation*}
		M_3\coloneqq\mathbb{C}E_{\alpha}^a\oplus\mathfrak{t}\oplus\bigoplus\limits_{n\geq 1}\mathfrak{g}_{-n\alpha}.
	\end{equation*}
The fact that the elements $E_{-\alpha}^a$ and $H_\alpha$ of $\mathfrak{s}_\alpha^a$ preserves $M_3$ is clear.
The only point that deserves some attention is the fact that the action of the element $E_{\alpha}^a\in\mathfrak{g}_\alpha^a$ maps $\mathfrak{g}_{-\alpha}$ into $M_3$. We have $[E_\alpha^a,\mathfrak{g}_{-\alpha}]\subseteq \mathfrak{g}_0$ while Lemma \ref{scentr} ensures that $\mathfrak{g}_0=\mathfrak{t}$ so that $[E_\alpha^a,\mathfrak{g}_{-\alpha}]\subseteq \mathfrak{t}\subset M_3$.

The element $H_\alpha$ acts by a traceless operator on $M_3$ so that
\begin{equation*}
	0=\alpha(H_\alpha)+0-\sum\limits_{n>0}n\alpha(H_\alpha)\mathrm{dim}\mathfrak{g}_{-n\alpha}
\end{equation*}
and we deduce that $\mathrm{dim}\mathfrak{g}_{-\alpha}=1$ while $\mathrm{dim}\mathfrak{g}_{-n\alpha}=0$ for every $n\geq2$.
\end{proof}

For later use, we include here the following lemma.

\begin{lemma}\label{trivialaction} The $1$-dimensional trivial representation of $\mathfrak{g}$ is the unique $1$-dimensional representation of $\mathfrak{g}$. 
\end{lemma}

\begin{proof}
	Let $V$ be a $1$-dimensional representation of $\mathfrak{g}$. From the representation theory of $\mathfrak{sl}(2,\mathbb{C})$, we know that every subalgebra $\mathfrak{s}_\alpha^a$ of $\mathfrak{g}$ acts trivially on $V$. On the other hand, the various subalgebras $\mathfrak{s}_\alpha^a$ generate $\mathfrak{g}$. We conclude that $\mathfrak{g}$ acts trivially on $V$.
\end{proof}

\section{Applications to finite-dimensional representations}

Let $\mathfrak{g}$ be a basic simple $\mathbb{Z}_2\times\mathbb{Z}_2$-graded Lie algebra with Cartan subalgebra $\mathfrak{t}$ and we assume that $\mathfrak{t}$ is self-centralizing in $\mathfrak{g}$.
In this section, based on the root theory developed in Section \ref{SEC2}, we provide a classification for finite-dimensional representations of $\mathfrak{g}$. More precisely, we prove a highest weight theorem for the finite-dimensional irreducible representations of $\mathfrak{g}$ (Subsection \ref{subsectionhwt}) while we show that every finite-dimensional representation of $\mathfrak{g}$ admits a direct sum decomposition into irreducible subrepresentations (Subsection \ref{weylscd}). Moreover, we show that there is an equivalence of categories between the category of finite-dimensional representations of $\mathfrak{g}$ and the category of finite-dimensional $\mathbb{Z}_2\times\mathbb{Z}_2$-graded representations of $\mathfrak{g}$.

We keep the same notation as in the previous sections. Namely, $\mathfrak{g}$ stands for a basic simple $\mathbb{Z}_2\times\mathbb{Z}_2$-graded Lie algebra with self-centralizing Cartan subalgebra $\mathfrak{t}$. We denote by $\Delta$ the root system of $\mathfrak{g}$ with respect to $\mathfrak{t}$ and by $\Delta^+$ a fixed positive root set in $\Delta$. For later use, set
\begin{align*}
	\mathfrak{n}&\coloneqq \bigoplus_{\alpha\in\Delta^+}\mathfrak{g}_\alpha,\\
	\mathfrak{n}^-&\coloneqq \bigoplus_{\alpha\in\Delta^+}\mathfrak{g}_{-\alpha}.
\end{align*}

\subsection{Highest Weight Theorem}\label{subsectionhwt}

In accordance with the terminology utilized in the theory of complex semisimple Lie algebras,
we say that an element $\lambda\in\mathfrak{t}^*$ 
is \textit{integral} if 
\begin{equation*}
	2\frac{\langle \lambda,\alpha\rangle}{\langle\alpha,\alpha\rangle}\in\mathbb{Z}, \quad \forall \alpha\in\Delta.
\end{equation*}
The element $\lambda$ is said to be \textit{dominant} if 
\begin{equation*}
	2\frac{\langle \lambda,\alpha\rangle}{\langle\alpha,\alpha\rangle}\geq0, \quad \forall \alpha\in\Delta^+.
\end{equation*}

\begin{theorem}\label{hwt}
Let $\mathfrak{g}$ be a basic simple $\mathbb{Z}_2\times\mathbb{Z}_2$-graded Lie algebra with self-centralizing Cartan subalgebra $\mathfrak{t}$. The set of equivalence classes of irreducible finite-dimensional representations of $\mathfrak{g}$ is parametrized 
by the dominant integral elements of $\mathfrak{t}^*$.
\end{theorem}

\begin{remark}
	We note that the finite-dimensional representations of $\mathfrak{g}$ that we consider in Theorem \ref{hwt} are not necessarily $\mathbb{Z}_2\times\mathbb{Z}_2$-graded. Nevertheless, we will see in Subsection \ref{weylscd} that every finite-dimensional representation of $\mathfrak{g}$ can be equipped with a $\mathbb{Z}_2\times\mathbb{Z}_2$-graded structure.
\end{remark}

The proof of Theorem \ref{hwt} is based on results from Section \ref{SEC2} and standard arguments utilized in the representation theory of complex simple Lie algebras (e.g., \cite[Sections V.2-3]{knappgreen}) that we shortly present in what follows.

Let $V $ be an irreducible finite-dimensional representation of $\mathfrak{g}$. According to Lemma \ref{oned}, for every $\alpha\in\Delta^+$, there is a unique $a\in\mathbb{Z}_2\times\mathbb{Z}_2$ such that $\mathfrak{g}_\alpha^a\neq\{0\}$, and hence, from Lemma \ref{pairing}, $\mathfrak{g}_{-\alpha}^a\neq\{0\}$. Fix the $\mathfrak{sl}_2$-triplet $\{h_\alpha,x_\alpha^a,y_\alpha^a\}$ (see Lemma \ref{sltriple}). From the representation theory of $\mathfrak{sl}(2,\mathbb{C})$, we deduce that every $h_\alpha$, $\alpha\in\Delta^+$, acts diagonally on $V$ so that $V$ is a weight-module in the sense that 
\begin{equation*}
	V=\bigoplus\limits_{\mu\in\mathfrak{t}^*}V_\mu
\end{equation*}
where
\begin{equation*}
	V_\mu:=\{v\in V\mid hv=\mu(h)v\text{ for all }h\in\mathfrak{t}\}
\end{equation*}
while every weight of $V$, i.e., every element $\mu\in\mathfrak{t}^*$ for which $V_\mu\neq0$, is 
integral. All the weights are real-valued when restricted to $\mathfrak{t}_0$ and, with respect to some lexicographic ordering (compatible with the choice of $\Delta^+$), there is a highest one, say $\lambda$, among them. The highest weight $\lambda$ of $V$ is dominant and all the weights of $V$ are of the form 
\begin{equation*}
	\lambda-\sum\limits_{\alpha\in\Delta^+}n_\alpha\alpha
\end{equation*}
for $n_\alpha$, $\alpha\in\Delta^+$, nonnegative integers, while the Weyl group $W_\mathfrak{g}$ of $\Delta$ permutes the weights of $V$.

It is worthwhile to mention that the universal enveloping algebra $U(\mathfrak{g})$ of $\mathfrak{g}$ admits a PBW basis \cite{scheunert,bautista}. The PBW theorem for color Lie algebras ensures that $U(\mathfrak{g})$ admits a triangular decomposition analogous to the classical case.
As a consequence, the highest weight subspace $V_\lambda$ of $M$ is $1$-dimensional. Indeed, if $v_\lambda$ is a highest weight vector of $V$, then we have
\begin{equation*}
	V=U(\mathfrak{g})v_\lambda=U(\mathfrak{n}^-)U(\mathfrak{t})U(\mathfrak{n})v_\lambda=U(\mathfrak{n}^-)v_\lambda
\end{equation*}
so that $V_\lambda$ is generated by the action of the degree zero elements of $U(\mathfrak{n}^-)$ on $v_\lambda$. Hence to every irreducible finite-dimensional representation of $\mathfrak{g}$ corresponds a dominant integral element of $\mathfrak{t}^*$, namely its highest weight.

For the inverse direction, based on Section \ref{SEC2}, one has to adapt the proof of \cite[Theorem 5.16]{knappgreen}.
Let $\lambda\in\mathfrak{t}^*$ be a dominant integral element and define the corresponding Verma module, i.e., the $\mathfrak{g}$-representation
\begin{equation*}
	M(\lambda)\coloneqq U(\mathfrak{g})\otimes_{U(\mathfrak{b})}\mathbb{C}_\lambda.
\end{equation*}
Here $\mathfrak{b}\coloneqq \mathfrak{t}\oplus\mathfrak{n}$ and $\mathbb{C}_\lambda$ is the $U(\mathfrak{b})$-representation where $\mathfrak{t}$ acts by $\lambda$ and $\mathfrak{n}$ acts trivially while $\mathfrak{g}$ acts on $M(\lambda)$ by left multiplication on the first compenent.
The Verma module $M(\lambda)$ contains a unique maximal proper $U(\mathfrak{g})$-submodule $N(\lambda)$ so that 
\begin{equation*}
	L(\lambda)\coloneqq M(\lambda)/N(\lambda)
\end{equation*}
is an irreducible highest weight $U(\mathfrak{g})$-module of highest weight $\lambda$. In particular, $N(\lambda)$ is the sum of all proper $U(\mathfrak{g})$-submodules of $M(\lambda)$. 

We are left to show that $L(\lambda)$ is finite-dimensional. Let $\alpha\in\mathfrak{t}^*$ be a simple root (with respect to $\Delta^+$) and set 
\begin{equation*}
	m\coloneqq 2\frac{\langle\lambda+\rho,\alpha\rangle}{\langle\alpha,\alpha\rangle}
\end{equation*}
where $\rho$ is the half-sum of positive roots of $\Delta$. The vector 
\begin{equation*}
	E_{-\alpha}^m(1\otimes1)
\end{equation*}
of $M(\lambda)$ is annihilated by the elements of $\mathfrak{n}$ (see \cite[Lemma 5.18]{knappgreen}) and, due to the PBW Theorem \cite{scheunert,bautista}, it generates a proper highest weight $\mathfrak{g}$-submodule in $M(\lambda)$. 
Then, if $v_\lambda$ is a nonzero highest weight vector of the $1$-dimensional highest weight subspace $L(\lambda)_\lambda$ of $L(\lambda)$, for all $n$ sufficiently large, we have $E_{-\alpha}^n v_\lambda=0$. According to Lemma \ref{oned}, there is a unique $a\in\mathbb{Z}_2\times\mathbb{Z}_2$ such that $\mathfrak{g}_{\alpha}^a$ and $\mathfrak{g}_{-\alpha}^a$ are nonzero.
Consider the standard $\mathfrak{sl}_2$-triplet $\{h_\alpha, x_\alpha^a,y_{\alpha}^a\}$ and let $\mathfrak{s}_\alpha^a$ be the corresponding Lie algebra, isomorphic to $\mathfrak{sl}(2,\mathbb{C})$. Let $U$ stand for the sum of all finite-dimensional $U(\mathfrak{s}_\alpha^a)$-submodules of $L(\lambda)$. Then $U$ is a nontrivial $\mathfrak{g}$-invariant subspace of $L(\lambda)$ so that $L(\lambda)=U$ \cite[p. 289]{knappgreen}. Let $\mu\in\mathfrak{t}^*$ be a weight of $L(\lambda)$ and $v_\mu\in L(\lambda)$ a nonzero weight vector for $\mu$. The representation theory of $\mathfrak{s}_\alpha^a\cong \mathfrak{sl}(2,\mathbb{C})$ ensures that, if $\langle\mu,\alpha\rangle>0$ then
 \begin{equation*}
 	(E_{-\alpha})^{\frac{2\langle\mu,\alpha\rangle}{\langle\alpha,\alpha\rangle}}v_\mu\neq0,
 \end{equation*} while if $\langle\mu,\alpha\rangle<0$ then  \begin{equation*}
 (E_\alpha)^{-\frac{2\langle\mu,\alpha\rangle}{\langle\alpha,\alpha\rangle}}v_\mu\neq0.
 \end{equation*}
  Hence, in both cases the element
\begin{equation*}
s_\alpha(\mu)=\mu-2\frac{\langle\mu,\alpha\rangle}{\langle\alpha,\alpha\rangle}
\end{equation*}
of $\mathfrak{t}^*$ is a weight of $L(\lambda)$. This shows that the Weyl group $W_\mathfrak{g}$ of $\Delta$ permutes the weights of $L(\lambda)$. Moreover, every weight of $L(\lambda)$ is $W_\mathfrak{g}$-conjugate to some dominant weight of $L(\lambda)$ and the dominant weights of $L(\lambda)$ are finitely many. We conlude that $L(\lambda)$ has finitely many weight subspaces while each of them must be finite-dimensional. Hence $L(\lambda)$ is an irreducible finite-dimensional representation of $\mathfrak{g}$ with highest weight $\lambda$. This concludes our discussion about Theorem \ref{hwt}.

\subsection{Theorem of complete reducibility}\label{weylscd}

The main goal of this subsection is to show the following theorem.

\begin{theorem}\label{complred}
	Let $\mathfrak{g}$ be a basic simple $\mathbb{Z}_2\times\mathbb{Z}_2$-graded Lie algebra with a self-centralizing Cartan subalgebra $\mathfrak{t}$, and $V$ a finite-dimensional representation of $\mathfrak{g}$. Then $V$ is completely reducible, i.e., there exist irreducible subrepresentations $V_1,\ldots,V_k$ of $V$ such that $V=V_1\oplus\ldots\oplus V_k$.
	\end{theorem}

We begin by defining the Casimir element of $\mathfrak{g}$ which is needed for proving an intermediate result. Let $\{Z_i\}$ be an orthonormal basis of $\mathfrak{g}$ with respect to the Killing form $K$ of $\mathfrak{g}$. We call \textit{the Casimir element of $\mathfrak{g}$} the element 
\begin{equation*}
	\Omega\coloneqq \sum\limits_{i}Z_i^2
\end{equation*} 
of the universal enveloping algebra $U(\mathfrak{g})$ of $\mathfrak{g}$. 
One can easily show that the Casimir element $\Omega$ of $\mathfrak{g}$ does not depend on the choice of the orthonormal basis $\{Z_i\}$ while it is homogeneous of degree $(0,0)$.

\begin{proposition}\label{centralelement}
	The Casimir element $\Omega$ of $\mathfrak{g}$ is a central element of $U(\mathfrak{g})$.
\end{proposition}

\begin{proof}
		It suffices to show that, for every $X\in\mathfrak{g}$, we have
	\begin{equation*}
		[X,\Omega]=0
	\end{equation*}
	in $U(\mathfrak{g})$. Since $\Omega$ is homogeneous, without loss of generality, we may assume that $X$ is homogeneous, i.e., $X\in\mathfrak{g}^a$ for some $a\in\mathbb{Z}_2\times\mathbb{Z}_2$. 
	
	Let $\{Z_i\}$ be an orthonormal basis of $\mathfrak{g}$ with respect to the Killing form $K$. According to the discussion before Definition \ref{defafterdisc}, we may assume that every $Z_i$ is homogeneous, say of degree $d_i$. The Casimir element $\Omega$ is
	\begin{equation*}
		\Omega=\sum_iZ_i^2.
	\end{equation*} 
Assume that 
	\begin{equation*}
		[X,Z_i]=\sum_ja_{ij}Z_j.
	\end{equation*}
Then 
\begin{align*}
	a_{ji}&=K([X,Z_j],Z_i)\\
	&=-(-1)^{ d_ja}K(Z_j,[X,Z_i])\\
	&=-(-1)^{d_ja} a_{ij}
\end{align*}
so that
\begin{align*}
	[X,\Omega]&=\sum_i[X,Z_i^2]\\
	&=\sum_i\big( [X,Z_i]Z_i+(-1)^{d_i a}Z_i[X,Z_i]\big)\\
	&=\sum_{i,j} a_{ij}Z_jZ_i+\sum_{i,j}(-1)^{d_i a}a_{ij}Z_iZ_j\\
	&=\sum_{i,j} a_{ij}Z_jZ_i+\sum_{i,j}(-1)^{d_j a}a_{ji}Z_jZ_i\\
	&=\sum_{i,j} a_{ij}Z_jZ_i-\sum_{i,j}(-1)^{2d_j a }a_{ij}Z_jZ_i\\
	&=0.
\end{align*}
\end{proof}

	Notice that because $\Omega$ is a $(0,0)$-degree element of $U(\mathfrak{g})$, we get
\begin{equation*}
X\Omega-\Omega X=[X,\Omega]=0. 
\end{equation*}
As a consequence, if $V$ is a finite-dimensional irreducible representation of $\mathfrak{g}$, then, due to Schur's Lemma \cite[Corollary 5.2]{knappgreen}, $\Omega$ acts by some scalar on $V$. More precisely, if we assume that the highest weight of $V$ is $\lambda\in\mathfrak{t}^*$,  $\Omega$ acts by 
\begin{equation*}
	\langle\lambda ,\lambda+2\rho\rangle 
\end{equation*}
on $V$. Recall that 
\begin{equation*}
	\rho\coloneqq \dfrac{1}{2}\sum\limits_{\alpha\in\Delta^+}\alpha.
\end{equation*} 
Indeed, let $H_i$ be an orthonormal basis of the Cartan subalgebra $\mathfrak{t}$ of $\mathfrak{g}$ and, for every $\alpha\in \Delta^+$ and suitable $a\in\mathbb{Z}_2\times\mathbb{Z}_2$, let $E_\alpha^a\in\mathfrak{g}_\alpha^a$ and $E_{-\alpha}^a\in\mathfrak{g}_{-\alpha}^a$ such that $K(E_\alpha^a,E_{-\alpha}^a)=1$. Due to Lemmata \ref{pairing} and \ref{oned}, we can always find such elements. Set $H_\alpha\coloneqq [E_\alpha^a,E_{-\alpha}^a]$. Then, the Casimir element $\Omega$ of $\mathfrak{g}$ is
\begin{align*}
	\Omega&=\sum_i H_i^2+\sum\limits_{\alpha\in\Delta^+}E_{\alpha}^aE_{-\alpha}^a+\sum\limits_{\alpha\in\Delta^+}E_{-\alpha}^aE_{\alpha}^a\\
	&=\sum_i H_i^2+\sum\limits_{\alpha\in\Delta^+}H_\alpha+2\sum\limits_{\alpha\in\Delta^+}E_{-\alpha}^aE_{\alpha}^a.
\end{align*}
Let $v_\lambda$ be a nonzero highest weight vector of $V$. Then 
\begin{align*}
	\Omega v_\lambda&=\big(\sum\limits_i \lambda(H_i)^2+\sum\limits_{\alpha\in\Delta^+} \lambda(H_\alpha)\big)v_\lambda\\
	&=\big(\langle\lambda,\lambda\rangle+\langle \lambda,2\rho\rangle\big)v_\lambda\\
	&=\langle\lambda ,\lambda+2\rho\rangle v_\lambda.
	\end{align*}
 In particular, the Casimir eigenvalue distinguishes non-isomorphic irreducible representations.
Since $\lambda$ is dominant, we deduce that $\Omega$ acts trivially on $V$ if and only if $\lambda=0$. In this case, according to Lemma \ref{trivialaction}, $V$ is the trivial representation of $\mathfrak{g}$.

Now, we are ready to say a few words about the proof of Theorem \ref{complred}. The argument is based on the results obtained in Section \ref{SEC2} and follows arguments from \cite[Section V.4]{knappgreen}. 

We will need the following intermediate result: if $V$ is a finite-dimensional representation of $\mathfrak{g}$ which contains a $\mathfrak{g}$-invariant subspace $U$ of codimension $1$, then there exists a $1$-dimensional $\mathfrak{g}$-invariant subspace $W$ of $V$ such that $V=W\oplus U$. If $U$ is $1$-dimensional, then $\mathfrak{g}$ acts trivially both on $U$ and $V/U$ (see Lemma \ref{trivialaction}) and the result follows. If $\dim U>1$ and $U$ is irreducible, we use the Casimir element $\Omega$ defined above. Namely, $\Omega(V)\subseteq U$ and $\Omega$ acts on $U$ by some nonzero scalar, so that $\dim \ker \Omega=1$ and $\ker \Omega \cap U=\{0\}$. Since $\Omega$ is central and of degree $(0,0)$ in $U(\mathfrak{g})$, it commutes with the action of $\mathfrak{g}$ on $V$ so that $\ker \Omega $ is $\mathfrak{g}$-invariant. Then $V=U\oplus \ker \Omega$. For the case when $U$ is nonirreducible, we apply induction on $\dim V$. The base case $\dim V=2$ is treated above. If $\dim V>2$, let $U_1$ be a $\mathfrak{g}$-invariant subspace of $U$ and form the quotients $U/U_1$ and $V/U_1$. Since $U/U_1$ is of codimension $1$ in $V/U_1$, from the induction hypothesis, there exists some $\mathfrak{g}$-invariant subspace $Y$ of $V$ such that 
\begin{equation*}
	V/U_1=U/U_1\oplus Y/U_1
\end{equation*}
and $\dim Y/U_1=1$. Since $\dim Y<\dim V$, there must be a $1$-dimensional $\mathfrak{g}$-invariant subspace $W$ of $Y$ such that $Y=U_1\oplus W$. Then \begin{equation*}
W\cap U=(W\cap Y)\cap U=W\cap(Y\cap U)\subseteq W\cap U_1=\{0\}
\end{equation*} and so $V=W\oplus U$.
For more details, see \cite[Lemma 5.30]{knappgreen}. 

We recall that a representation $V$ of $\mathfrak{g}$ is said to be $\mathbb{Z}_2\times\mathbb{Z}_2$-graded
whenever it admits a vector space decomposition
\begin{equation*}
	V=\bigoplus\limits_{b\in\mathbb{Z}_2\times\mathbb{Z}_2}V^b
\end{equation*}
such that 
\begin{equation*}
	\mathfrak{g}^aV^b\subseteq V^{a+b}
\end{equation*}
for every $a,b\in\mathbb{Z}_2\times\mathbb{Z}_2$.
One checks that the above intermediate result holds in its $\mathbb{Z}_2\times\mathbb{Z}_2$-graded version as well. More precisely, if $V$ is a $\mathbb{Z}_2\times\mathbb{Z}_2$-graded representation of $\mathfrak{g}$ which contains a $\mathfrak{g}$-invariant $\mathbb{Z}_2\times\mathbb{Z}_2$-graded subspace $U$ of codimension $1$, then there exists a $1$-dimensional $\mathfrak{g}$-invariant $\mathbb{Z}_2\times\mathbb{Z}_2$-graded subspace $W$ such that $V=W\oplus U$. This is due to \cite[Corollary 2.3.4]{graded} which states that a $\mathbb{Z}_2\times\mathbb{Z}_2$-graded $\mathfrak{g}$-submodule $N$ of some $\mathbb{Z}_2\times\mathbb{Z}_2$-graded $\mathfrak{g}$-module $M$ is a direct summand of $M$ in the category of $\mathfrak{g}$-modules if and only if it is a direct summand of $M$ in the category of $\mathbb{Z}_2\times\mathbb{Z}_2$-graded $\mathfrak{g}$-modules.

Let us now prove Theorem \ref{complred} in the case where the representation $(\varphi,V)$ of $\mathfrak{g}$ is $\mathbb{Z}_2\times\mathbb{Z}_2$-graded.
Assume that $V$ is nonirreducible and
let $U$ be a nonzero $\mathfrak{g}$-invariant $\mathbb{Z}_2\times\mathbb{Z}_2$-graded subspace of $V$. It suffices to show that there exists a $\mathfrak{g}$-invariant $\mathbb{Z}_2\times\mathbb{Z}_2$-graded complement of $U$ in $V$.
Set
 \begin{equation*}
 	Z:=\{\gamma \in \mathrm{End}(V)\mid \gamma(V)\subseteq U\text{ and } \gamma_{\vert U} \text{ is scalar}\}
 \end{equation*}
and define the linear map
\begin{equation*}
	\sigma: \mathfrak{g}\rightarrow \mathrm{End}\big(\mathrm{End}(V)\big)
\end{equation*}
given by 
\begin{equation*}
	\sigma(X)\gamma\coloneqq \varphi(X)\gamma -(-1)^{\vert X\vert\vert \gamma\vert}\gamma\varphi(X)
\end{equation*}
for homogeneous $X\in\mathfrak{g}$ and $\gamma\in\mathrm{End}(V)$.
Then $\sigma$ is a $\mathbb{Z}_2\times\mathbb{Z}_2$-graded representation of $\mathfrak{g}$ on $\mathrm{End}(V)$ while it preserves the $\mathbb{Z}_2\times\mathbb{Z}_2$-graded subspace $Z$. Let us be more precise about the last statement. Let $S$ be a $\mathbb{Z}_2\times\mathbb{Z}_2$-graded subspace of $V$ such that $V=U\oplus S$. Then every $\gamma\in Z$ can be written as a linear combination
\begin{equation*}
	\gamma=\gamma_1+\sum\limits_{a\in\mathbb{Z}_2\times\mathbb{Z}_2} \gamma_2^a
\end{equation*} 
where $\gamma_1$ is a degree $(0,0)$ operator with $(\gamma_1)_{\vert U}=\gamma_{\vert U}$, i.e., scalar on $U$, and $(\gamma_1)_{\vert S}=0$, while every $\gamma_2^a$, $a\in \mathbb{Z}_2\times\mathbb{Z}_2$, is a homogeneous operator of degree $a$ with $\gamma_2^a(U)=\{0\}$ and $\gamma_2^a(S)\subseteq U$. Then, for $X\in\mathfrak{g}$ and $v\in V$, one can easily check that 
\begin{equation*}
\big(	\sigma(X)\gamma\big)(v)\in U.
\end{equation*}
Let us, now, show that $\big(\sigma(X)\gamma\big)_{\vert U}$ is scalar. For $u\in U$ and homogeneous $X\in\mathfrak{g}$, we have
\begin{align*}
	\sigma(X)\gamma(u)=&\sigma(X)(\gamma_1+\sum_a\gamma_2^a)(u)\\
	=&\sigma(X)\gamma_1(u)+\sum_a\sigma(X)\gamma_2^a(u)\\
	=&\varphi(X)\gamma_1(u)-\gamma_1\varphi(X)(u)\\
	&+\sum_a \varphi(X)\gamma_2^a(u)-(-1)^{\vert X\vert a}\gamma_2^a\varphi(X)(u).
\end{align*}
Since $u\in U$ and $U$ is $\varphi(\mathfrak{g})$-invariant, we have $\varphi(X)(u)\in U$ so that 
\begin{equation*}
	\varphi(X)\gamma_1(u)=\lambda \varphi(X)(u)
\end{equation*} and \begin{equation*}\gamma_1\varphi(X)(u)=\lambda \varphi(X)(u)
\end{equation*}
 for some $\lambda\in \mathbb{C}$. Moreover, by definition, $\gamma_2^a$'s are trivial on $U$ so that $\gamma_2^a(u)=0$ and $\gamma_2^a\varphi(X)(u)=0$ for every $a\in\mathbb{Z}_2\times\mathbb{Z}_2$. As a consequence, we have 
 \begin{equation*}
 	\sigma(X)\gamma(u)=0.
 \end{equation*}
In other words, $\big(\sigma(X)\gamma\big)_{\vert U}=0$ for every $\gamma\in Z$.
By the same argument, we see that the $\mathbb{Z}_2\times\mathbb{Z}_2$-graded subspace
\begin{equation*}
	Z'\coloneqq \{\gamma\in Z\vert \gamma_{\vert U}=0\}
\end{equation*}
is $\mathfrak{g}$-invariant, too, while $Z/Z'$
is $1$-dimensional. Then, by the above intermediate result, there must be a $1$-dimensional $\mathfrak{g}$-invariant $\mathbb{Z}_2\times\mathbb{Z}_2$-graded subspace $W=\mathbb{C}\gamma$ of $Z$ such that $Z=Z'\oplus W$. Of course, $\gamma$ must be homogeneous while its restriction to $U$ is a nonzero scalar operator. We deduce that $\gamma$ must be of degree $(0,0)$. Since $W$ is $1$-dimensional, $\mathfrak{g}$ acts on $W$ trivially (see Lemma \ref{trivialaction}) so that $\gamma$ commutes with $\varphi(X)$ for every $X\in\mathfrak{g}$. Therefore, $\ker \gamma$ is a $\mathfrak{g}$-invariant $\mathbb{Z}_2\times\mathbb{Z}_2$-graded subspace of $V$ with $V=U\oplus \ker \gamma$. This proves the $\mathbb{Z}_2\times\mathbb{Z}_2$-graded version of Theorem \ref{hwt}.
	
Let us now prove Theorem \ref{complred} for a general finite-dimensional representation $V$ of $\mathfrak{g}$, i.e., not necessarily $\mathbb{Z}_2\times\mathbb{Z}_2$-graded. Actually, we will show that any such representation admits a $\mathbb{Z}_2\times\mathbb{Z}_2$-graded $\mathfrak{g}$-representation structure. Moreover, with respect to this $\mathbb{Z}_2\times\mathbb{Z}_2$-grading, any $\mathfrak{g}$-invariant subspace of $V$ turns out to be $\mathbb{Z}_2\times\mathbb{Z}_2$-graded. As a consequence, we obtain Theorem \ref{complred}.

Let $V$ be a finite-dimensional representation of $\mathfrak{g}$.
In what follows, we will describe how one can equip $V$ with a $\mathbb{Z}_2\times\mathbb{Z}_2$-graded $\mathfrak{g}$-representation structure. We start by noting that, from the representation theory of $\mathfrak{sl}(2,\mathbb{C})$ and by using Lemma \ref{sltriple}, we deduce that $V$ is a weight module, i.e., it admits a weight subspace decomposition 
\begin{equation*}
	V=\bigoplus\limits_{\mu\in\mathfrak{t}^*}V_\mu.
\end{equation*}
If $\Lambda$ stands for the root lattice $\mathbb{Z}\Delta$ in $\mathfrak{t}^*$, then the weights of $V$ lie in a disjoint union of the form 
\begin{equation}\label{disjointunion}
	\bigsqcup\limits_{j=1}^k\lambda_j+\Lambda
\end{equation}
for suitable weights $\lambda_i$ of $V$. Since every component of \eqref{disjointunion} gives a direct summand of $V$, we may assume that all the weights of $V$ lie in a unique component, say $\lambda+\Lambda$ for some fixed weight $\lambda\in\mathfrak{t}^*$ of $V$. We may assume that $\lambda$ is maximal among all the weights of $V$. Any other weight $\mu$ of $V$ is written in a unique way as
\begin{equation*}
	\mu=\lambda+\sum\limits_{i}n_i\alpha_i
\end{equation*}
where $\Pi\coloneqq \{\alpha_1,\ldots,\alpha_l\}$ is the simple root set of $\Delta$ (related to $\Delta^+$) and $n_i$ are integers. 
Let us recall that according to Lemma \ref{oned}, for every $\beta\in\Delta$, $\dim\mathfrak{g}_\beta=1$ so that there exists a unique $b\in\mathbb{Z}_2\times\mathbb{Z}_2$ such that $\mathfrak{g}_\beta\subset \mathfrak{g}^b$. By abuse of notation, let us write $\vert \beta\vert \coloneqq b$. We note that, according to Lemma \ref{pairing}, $\vert \beta\vert=\vert-\beta\vert=-\vert \beta\vert$.
We set the degree $\vert V_\mu\vert$ of the weight subspace $V_\mu$ to be
\begin{equation*}
\vert V_\mu\vert\coloneqq\sum\limits_{i}n_i \vert \alpha_i\vert.
	\end{equation*}
Then $V$ becomes a $\mathbb{Z}_2\times\mathbb{Z}_2$-graded vector space.

Let us, now, show that, with respect to this grading, $V$ becomes a $\mathbb{Z}_2\times\mathbb{Z}_2$-graded representation of $\mathfrak{g}$. In other words, we will show that for every $\beta\in\Delta$, whenever $\mathfrak{g}_\beta V_\mu$ is nontrivial, we have 
\begin{equation*}
	\vert \mathfrak{g}_{\beta} V_\mu\vert=\vert V_\mu\vert+\vert \beta\vert.
\end{equation*}
Let us assume that $\beta$ is positive; in a similar way one can show the statement for negative $\beta$.
Write $\beta$ as a nonnegative integer sum of simple roots
\begin{equation*}
	\beta=\sum\limits_{i}m_i\alpha.
\end{equation*}
Then 
\begin{equation*}
	\mathfrak{g}_{\beta} V_\mu\subseteq V_{\mu+\beta}
\end{equation*}
so that
\begin{equation}\label{grinclusion}
	\vert \mathfrak{g}_{\beta}V_\mu\vert =\sum_in_i\vert\alpha_i\vert + \sum_im_i\vert \alpha_i\vert=\vert V_\mu\vert+\sum_im_i\vert \alpha_i\vert
\end{equation}
We will have finished if we prove the following lemma.

\begin{lemma} If $\beta\in\Delta^+$ with 
	\begin{equation*}
		\beta=\sum_im_i\alpha_i,
	\end{equation*}
then 
\begin{equation*}
	\vert \beta\vert =\sum_im_i\vert \alpha_i\vert.
\end{equation*}
\end{lemma}

\begin{proof} Let $\beta\in\Delta^+$ be as in the statement. We will apply an induction on the length 
	\begin{equation*}
		l(\beta)\coloneqq\sum_i m_i 
	\end{equation*}
of $\beta$. If $l(\beta)=1$, i.e., $\beta$ is a simple root, then the statement holds trivially. 
 Let us assume that the statement holds for every root $\gamma$ with $l(\gamma)<N$ and $l(\beta)=N$. 
 Recall that, according to Corollary \ref{abstractrootsystem}, $\Delta$ is an abstract root system so that, if 
 \begin{equation*}
 	\rho\coloneqq\dfrac{1}{2} \sum\limits_{\alpha\in\Delta^+}\alpha, 
 \end{equation*}
 then $\langle\alpha_i,\rho\rangle >0$ for every $\alpha_i\in\Pi$ (e.g., \cite[Proposition 2.69]{knappgreen}). 
 Therefore $\langle\beta,\rho\rangle>0$ and there must be some simple root $\alpha_j\in\Pi$ such that $\langle\beta,\alpha_j\rangle >0$. 
 The representation theory of the $\mathfrak{sl}_2$-copy 
 \begin{equation*}
 \mathfrak{s}_{\alpha_j}^{\vert \alpha_j\vert}\coloneqq\mathbb{C}\{h_{\alpha_j},x_{\alpha_j}^{\vert \alpha_j\vert},y_{\alpha_j}^{\vert \alpha_j\vert}\}
 \end{equation*}
 ensures that 
 \begin{equation*}
 	[\mathfrak{g}_{\beta},\mathfrak{g}_{-\alpha_j}]\neq0
 \end{equation*}
so that, from Lemma \ref{oned}, 
\begin{equation*}
	[\mathfrak{g}_\beta,\mathfrak{g}_{-\alpha_j}]=\mathfrak{g}_{\beta-\alpha_j},
\end{equation*}
and 
\begin{equation*}
	[\mathfrak{g}_{\alpha_j},\mathfrak{g}_{\beta-\alpha_j}]=\mathfrak{g}_\beta.
\end{equation*}
As a consequence, 
\begin{equation*}
	\vert\beta\vert=\vert \beta-\alpha_j\vert +\vert \alpha_j\vert.
\end{equation*}
From the induction hypothesis, 
\begin{equation*}
	\vert \beta-\alpha_j\vert=\sum_{i\neq j}m_i\vert\alpha_j\vert +(m_j-1)\vert\alpha_j\vert
\end{equation*}
so that 
\begin{equation*}
	\vert\beta\vert =\sum\limits_im_i\vert\alpha_j\vert.
\end{equation*}

\end{proof}

As a consequence of the above Lemma, relation \eqref{grinclusion} implies that 
\begin{equation*}
	\vert \mathfrak{g}_{\pm\beta}V_\mu\vert =\vert \beta\vert +\vert V_\mu\vert
\end{equation*}
and $V$ is a $\mathbb{Z}_2\times \mathbb{Z}_2$-graded $\mathfrak{g}$-representation.

Let $U$ be a $\mathfrak{g}$-invariant subspace of $V$. Then $U$ is a $\mathbb{Z}_2\times\mathbb{Z}_2$-graded subspace of $V$. Indeed, $U$ is a weight module 
\begin{equation*}
	U=\bigoplus\limits_{\mu\in\mathfrak{t}^*}U_\mu
\end{equation*}
and $U_\mu\subseteq V_\mu$. Since $V_\mu$ consists of homogeneous elements, this is also the case for $U_\mu$. We deduce that $U$ is a $\mathbb{Z}_2\times\mathbb{Z}_2$-graded $\mathfrak{g}$-invariant subspace of $V$. 

As a consequence of the above discussion, we get the following proposition.

\begin{proposition}
	There is an equivalence of categories between the category of finite-dimensional representations of $\mathfrak{g}$ and the category of finite-dimensional $\mathbb{Z}_2\times\mathbb{Z}_2$-graded representations of $\mathfrak{g}$.
\end{proposition}

\section{Two examples and a question}
In this section, we illustrate the preceding results with two examples and conclude by raising a question.

Let $\mathfrak{g}$ be the $\mathbb{Z}_2\times\mathbb{Z}_2$-graded Lie algebra $\mathfrak{so}(p,q,r,s)_\mathbb{C}$ defined in \cite{stoil25a} by
\begin{equation}\label{so}
	\begin{array}{c c}
		{\begin{array} {@{} c c  cc @{}} \ \ p\ \ \ & \ \ q\ \ \ & \ \ r\ \ \ & \ \ s \ \ \end{array}} & {} \\  
		\left(\begin{array}{cccc} 
			a_{(0,0)} & a_{(0,1)} & a_{(1,0)} & a_{(1,1)} \\[1mm] 
			-a_{(0,1)}^t & b_{(0,0)} & b_{(1,1)} & b_{(1,0)} \\[1mm] 
			-a_{(1,0)}^t & b_{(1,1)}^t & c_{(0,0)} & c_{(0,1)} \\[1mm] 
			-a_{(1,1)}^t & b_{(1,0)}^t & c_{(0,1)}^t & d_{(0,0)} 
		\end{array}\right) 
		& \hspace{-2mm}	{\begin{array}{l}
				p \\[1mm]  q \\[1mm]	r \\[1mm] s
		\end{array} }\\ 
	\end{array} 
\end{equation}
where $a_{(0,0)}$, $b_{(0,0)}$, $c_{(0,0)}$ and $d_{(0,0)}$ are antisymmetric matrices, and the indices on each block indicate the corresponding grading. Observe that the form \eqref{so} is similar to the standard form of matrices in the Lie algebra $\mathfrak{so}(p+q+r+s,\mathbb{C})$ apart from a sign difference in the blocks $b_{(1,1)}^t$, $b_{(1,0)}^t$ and $c_{(0,1)}^t$.

\subsection{The case of $\mathfrak{so}(4,2,2,2)_\mathbb{C}$} Assume that $\mathfrak{g}$ is $\mathfrak{so}(4,2,2,2)$.
The Cartan subalgebra $\mathfrak{t}$ of $\mathfrak{g}$ is generated by the matrices
\begin{equation*}
H_1\coloneqq	\begin{pmatrix}
		\begin{pmatrix}\begin{matrix}
			0&i\\
			-i&0
		\end{matrix}&\\
	&	0_{2\times2}
\end{pmatrix}&&&\\
&	0_{2\times2}&&\\
&&0_{2\times2}&\\
&&&0_{2\times2}
	\end{pmatrix}
\end{equation*}
\begin{equation*}
	H_2\coloneqq
	\begin{pmatrix}
		\begin{pmatrix}0_{2\times2}&\\
			&	\begin{matrix}
				0&i\\
				-i&0
			\end{matrix}
		\end{pmatrix}&&&\\
		&	0_{2\times2}&&\\
		&&0_{2\times2}&\\
		&&&0_{2\times2}
	\end{pmatrix}
\end{equation*}
\begin{equation*}
H_3\coloneqq\begin{pmatrix}
	\begin{pmatrix}0_{2\times2}&\\
		&	0_{2\times2}
	\end{pmatrix}&&&\\
	&	\begin{matrix}
		0&i\\
		-i&0
	\end{matrix}&&\\
	&&0_{2\times2}&\\
	&&&0_{2\times2}
\end{pmatrix}
\end{equation*}
\begin{equation*}
	H_4\coloneqq\begin{pmatrix}
		\begin{pmatrix}0_{2\times2}&\\
			&	0_{2\times2}
		\end{pmatrix}&&&\\
		&	0_{2\times2}&&\\
		&&\begin{matrix}
			0&i\\
			-i&0
		\end{matrix}&\\
		&&&0_{2\times2}
	\end{pmatrix}
\end{equation*}
\begin{equation*}
	H_5\coloneqq\begin{pmatrix}
		\begin{pmatrix}0_{2\times2}&\\
			&	0_{2\times2}
		\end{pmatrix}&&&\\
		&	0_{2\times2}&&\\
		&&0_{2\times2}&\\
		&&&\begin{matrix}
			0&i\\
			-i&0
		\end{matrix}
	\end{pmatrix}.
\end{equation*}
In other words, $\mathfrak{h}$ coincides with the Cartan subalgebra of the corresponding Lie algebra $\mathfrak{so}(10,\mathbb{C})$. In this case, $\mathfrak{t}$ is self-centralizing.

If we define $\varepsilon_i\in\mathfrak{t}^*$, $i\in{1,\ldots,5}$, by
\begin{equation*}
	\varepsilon_i(H_j)=\delta_{ij}
\end{equation*}
then the root system of $\mathfrak{g}$ with respect to $\mathfrak{h}$ is 
\begin{equation*}
	\Delta=\{\pm\varepsilon_i\pm\varepsilon_j\vert 1\leq i\neq j \leq5\}. 
\end{equation*}

For $1\leq i<j\leq 5$, set
\begin{equation*}
\begin{matrix}
	X_{\varepsilon_i-\varepsilon_j}\coloneqq \begin{pmatrix}1&i\\
		-i&1
		\end{pmatrix},&X_{\varepsilon_i+\varepsilon_j}\coloneqq \begin{pmatrix}1&-i\\
		-i&-1
	\end{pmatrix},\\
	X_{-\varepsilon_i+\varepsilon_j}\coloneqq \begin{pmatrix}1&-i\\
		i&1
	\end{pmatrix},&X_{-\varepsilon_i-\varepsilon_j}\coloneqq \begin{pmatrix}1&i\\
	i&-1
\end{pmatrix}
	\end{matrix}
\end{equation*}
Then, for $1\leq i<j\leq 5$, the root vectors are
\begin{itemize}
	\item If $i=1$ or $i=2$,
\begin{equation*}
	E_{\pm\varepsilon_i\pm\varepsilon_j}\coloneqq \begin{blockarray}{ccc}
		i & j  \\
		\begin{block}{(cc)c}
			0 & X_{\pm\varepsilon_i\pm\varepsilon_j}  & i \\
			-X_{\pm\varepsilon_i\pm\varepsilon_j}^t & 0 & j \\
		\end{block}
	\end{blockarray}
\end{equation*}
\item If $i=3$ or $i=4$,
\begin{equation*}
	E_{\pm\varepsilon_i\pm\varepsilon_j}\coloneqq \begin{blockarray}{ccc}
		i & j  \\
		\begin{block}{(cc)c}
			0 & X_{\pm\varepsilon_i\pm\varepsilon_j}  & i \\
			X_{\pm\varepsilon_i\pm\varepsilon_j}^t & 0 & j \\
		\end{block}
	\end{blockarray}
\end{equation*}
\end{itemize}
For every $\alpha\in\Delta$, the corresponding root subspace $\mathfrak{g}_\alpha$ is $1$-dimensional. The gradings of each root vector is 
\begin{itemize}
	\item $(0,0)$ for $E_{\pm\varepsilon_1\pm\varepsilon_2}$,
	\item $(0,1)$ for $E_{\pm\varepsilon_1\pm\varepsilon_3}$, $E_{\pm\varepsilon_2\pm\varepsilon_3}$, $E_{\pm\varepsilon_4\pm\varepsilon_5}$,
	\item  $(1,0)$ for $E_{\pm\varepsilon_1\pm\varepsilon_4}$, $E_{\pm\varepsilon_2\pm\varepsilon_4}$, $E_{\pm\varepsilon_3\pm\varepsilon_5}$,
	\item $(1,1)$ for $E_{\pm\varepsilon_1\pm\varepsilon_5}$, $E_{\pm\varepsilon_2\pm\varepsilon_5}$, $E_{\pm\varepsilon_3\pm\varepsilon_4}$.
\end{itemize}

The similarities with the Lie algebra $\mathfrak{so}(10,\mathbb{C})$ are striking.
 This is expected for the following reason: There is a canonical vector space isomorphism 
 \begin{equation*}
 	\varphi:\mathfrak{g}\rightarrow\mathfrak{so}(10,\mathbb{C})
 	\end{equation*}
 	  which is completely determined by being the identity map on the upper blocks 
 	  \begin{equation*}
 	  	\begin{array}{c c}
 	  		{\begin{array} {@{} c c  cc @{}} \ \ 4\ \ \ & \ \ 2\ \ \ & \ \ 2\ \ \ & \ \ 2 \ \ \end{array}} & {} \\  
 	  		\left(\begin{array}{cccc} 
 	  			a_{(0,0)} & a_{(0,1)} & a_{(1,0)} & a_{(1,1)} \\[1mm] 
 	  			& b_{(0,0)} & b_{(1,1)} & b_{(1,0)} \\[1mm] 
 	  			&  & c_{(0,0)} & c_{(0,1)} \\[1mm] 
 	  			 &  & & d_{(0,0)} 
 	  		\end{array}\right) 
 	  		& \hspace{-2mm}	{\begin{array}{l}
 	  				4 \\[1mm]  2 \\[1mm]	2 \\[1mm] 2
 	  		\end{array} }\\ 
 	  	\end{array} 
 	  \end{equation*}
 	  Moreover, the restriction of $\varphi$ to 
 	 \begin{equation*}
 	 \mathfrak{g}^{(0,0)}\coloneqq	\begin{array}{c c}
 	 		{\begin{array} {@{} c c  cc @{}} \ \ 4\ \ \ & \ \ 2\ \ \ & \ \ 2\ \ \ & \ \ 2 \ \ \end{array}} & {} \\  
 	 		\left(\begin{array}{cccc} 
 	 			a_{(0,0)} & &  &  \\[1mm] 
 	 			& b_{(0,0)} &  &  \\[1mm] 
 	 			&  & c_{(0,0)} &  \\[1mm] 
 	 			&  & & d_{(0,0)} 
 	 		\end{array}\right) 
 	 		& \hspace{-2mm}	{\begin{array}{l}
 	 				4 \\[1mm]  2 \\[1mm]	2 \\[1mm] 2
 	 		\end{array} }\\ 
 	 	\end{array} 
 	 \end{equation*}
 	  is a Lie algebra isomorphism onto its image $\varphi(\mathfrak{g}^{(0,0)})$. Then $\mathfrak{g}^{(0,0)}$ and $\varphi(\mathfrak{g}^{(0,0)})$ act on $\mathfrak{g}$ and $\mathfrak{so}(10,\mathbb{C})$ respectively and the map $\varphi$ is equivariant with respect to these actions while these actions in some sense determine the root systems of $\mathfrak{g}$ and $\mathfrak{so}(10,\mathbb{C})$.
 	  
 	  \begin{remark}
 	  	This example is typical for the case where $p$, $q$ and $r$ are even. 
 	  \end{remark}
 	  
 	  \subsection{The case of $\mathfrak{so}(4,2,1,1)_{\mathbb{C}}$}
 	  Let us assume that $\mathfrak{g}$ is $\mathfrak{so}(4,2,1,1)_\mathbb{C}$. The Cartan subalgebra $\mathfrak{t}$ of $\mathfrak{g}$ is generated by the matrices 
 	  \begin{equation*}
 H_1\coloneqq	\begin{pmatrix}
 	\begin{pmatrix}\begin{matrix}
 			0&i\\
 			-i&0
 		\end{matrix}&\\
 		&	0_{2\times2}
 	\end{pmatrix}&&&\\
 	&	0_{2\times2}&&\\
 	&&0&\\
 	&&&0
 \end{pmatrix}
\end{equation*}
\begin{equation*}
H_2\coloneqq
\begin{pmatrix}
	\begin{pmatrix}0_{2\times2}&\\
		&	\begin{matrix}
			0&i\\
			-i&0
		\end{matrix}
	\end{pmatrix}&&&\\
	&	0_{2\times2}&&\\
	&&0&\\
	&&&0
\end{pmatrix}
\end{equation*}
\begin{equation*}
H_3\coloneqq\begin{pmatrix}
	\begin{pmatrix}0_{2\times2}&\\
		&	0_{2\times2}
	\end{pmatrix}&&&\\
	&	\begin{matrix}
		0&i\\
		-i&0
	\end{matrix}&&\\
	&&0&\\
	&&&0
\end{pmatrix}
\end{equation*}
Keeping the same notation for $\varepsilon_i$, $1\leq i\leq3$, as above, the root system with respect to $\mathfrak{t}$ is 
\begin{equation*}
	\Delta=\{\pm\varepsilon_i\pm\varepsilon_j\vert 1\leq i<j\leq 3\}\sqcup\{\pm\varepsilon_i\vert1\leq i\leq 3\}.
\end{equation*}
The root vectors (with the corresponding gradings) are 
	\begin{equation*}
	E_{\pm\varepsilon_1\pm\varepsilon_2}^{(0,0)}\coloneqq \begin{pmatrix}
		\begin{pmatrix}
			0_{2\times2}&X_{\pm\varepsilon_1\pm\varepsilon_2}\\
			-X_{\pm\varepsilon_1\pm\varepsilon_2}^t&0_{2\times2}
		\end{pmatrix}&
	\begin{matrix}0_{2\times2}\\
		0_{2\times2}
		\end{matrix}&&\\
		\begin{matrix}0_{2\times2}&0_{2\times2}
			\end{matrix}&0_{2\times2}	&&\\
		&&0&\\
		&&&0
	\end{pmatrix}
\end{equation*}
\begin{equation*}
	E_{\pm\varepsilon_1\pm\varepsilon_3}^{(0,1)}\coloneqq \begin{pmatrix}
		\begin{pmatrix}
			0_{2\times2}&0_{2\times2}\\
			0_{2\times2}&0_{2\times2}
		\end{pmatrix}&
	\begin{matrix}
		X_{\pm\varepsilon_1\pm\varepsilon_3}\\
		0_{2\times2}
		\end{matrix}
		&&\\
		\begin{matrix}
		-X_{\pm\varepsilon_1\pm\varepsilon_3}^t &0_{2\times2}
		\end{matrix}&0_{2\times2}	&&\\
		&&0&\\
		&&&0
	\end{pmatrix}
\end{equation*}
\begin{equation*}
	E_{\pm\varepsilon_2\pm\varepsilon_3}^{(0,1)}\coloneqq \begin{pmatrix}
		\begin{pmatrix}
			0_{2\times2}&0_{2\times2}\\
			0_{2\times2}&0_{2\times2}
		\end{pmatrix}&
		\begin{matrix}
			0_{2\times2}\\
			X_{\pm\varepsilon_2\pm\varepsilon_3}
		\end{matrix}
		&&\\
		\begin{matrix}
			0&-X_{\pm\varepsilon_2\pm\varepsilon_3}^t 
		\end{matrix}&0_{2\times2}	&&\\
		&&0&\\
		&&&0
	\end{pmatrix}
\end{equation*}
\begin{equation*}
X_{\pm\varepsilon_1}^{(1,0)}\coloneqq\begin{pmatrix}
0_{6\times6}&\begin{matrix}
1&0\\
\mp i&0\\
0&0\\
0&0\\
0&0\\
0&0
\end{matrix}\\
\begin{matrix}
	-1&\pm i&0&0&0&0\\
	0&0&0&0&0&0
\end{matrix}&0_{2\times2}
\end{pmatrix}
\end{equation*}
\begin{equation*}
	E_{\pm\varepsilon_1}^{(1,1)}\coloneqq\begin{pmatrix}
		0_{6\times6}&\begin{matrix}
			0&1\\
			0&\mp i\\
			0&0\\
			0&0\\
			0&0\\
			0&0
		\end{matrix}\\
		\begin{matrix}
			0&0&0&0&0&0\\
			-1&\pm i&0&0&0&0
		\end{matrix}&0_{2\times2}
	\end{pmatrix}
\end{equation*}
\begin{equation*}
	E_{\pm\varepsilon_2}^{(1,0)}\coloneqq\begin{pmatrix}
		0_{6\times6}&\begin{matrix}
			0&0\\
			0&0\\
			1&0\\
			\mp i&0\\
			0&0\\
			0&0
		\end{matrix}\\
		\begin{matrix}
			0&0&-1&\pm i&0&0\\
			0&0&0&0&0&0
		\end{matrix}&0_{2\times2}
	\end{pmatrix}
\end{equation*}
\begin{equation*}
	E_{\pm\varepsilon_2}^{(1,1)}\coloneqq\begin{pmatrix}
		0_{6\times6}&\begin{matrix}
			0&0\\
			0&0\\
			0&1\\
			0&\mp i\\
			0&0\\
			0&0
		\end{matrix}\\
		\begin{matrix}
			0&0&0&0&0&0\\
			0&0&-1&\pm i&0&0
		\end{matrix}&0_{2\times2}
	\end{pmatrix}
\end{equation*}
\begin{equation*}
	E_{\pm\varepsilon_3}^{(1,0)}\coloneqq\begin{pmatrix}
		0_{6\times6}&\begin{matrix}
			0&0\\
			0&0\\
			0&0\\
			0&0\\
			1&0\\
			\mp i&0
		\end{matrix}\\
		\begin{matrix}
			0&0&0&0&1&\mp i\\
			0&0&0&0&0&0
		\end{matrix}&0_{2\times2}
	\end{pmatrix}
\end{equation*}
\begin{equation*}
	E_{\pm\varepsilon_3}^{(1,1)}\coloneqq\begin{pmatrix}
		0_{6\times6}&\begin{matrix}
			0&0\\
			0&0\\
			0&0\\
			0&0\\
			0&1\\
			0&\mp i
		\end{matrix}\\
		\begin{matrix}
			0&0&0&0&0&0\\
			0&0&0&0&1&\mp i
		\end{matrix}&0_{2\times2}
	\end{pmatrix}
\end{equation*}
Finally, we have the vector 
\begin{equation*}
	E_{0}^{(0,1)}\coloneqq \begin{pmatrix}
		0_{6\times 6}&0_{6\times 2}\\
		0_{2\times 6}&
		\begin{matrix}
			0&1\\
			1&0
			\end{matrix}
	\end{pmatrix}
\end{equation*}
which is annihilated by $\mathfrak{t}$. 

From the above discussion we see that for a fixed grading degree $a$ and a fixed root $\alpha$ the corresponding root subspace $\mathfrak{g}^a_\alpha$ is at most $1$-dimensional. This is what Lemma \ref{onedim} indicates. On the other hand, we see that the root subspaces $\mathfrak{g}_{\pm\varepsilon_1}$, $\mathfrak{g}_{\pm\varepsilon_2}$ and $\mathfrak{g}_{\pm\varepsilon_3}$ are $2$-dimensional, namely
\begin{equation*}
	\mathfrak{g}_{\pm\varepsilon_i}=\mathfrak{g}_{\pm\varepsilon_i}^{(1,0)}\oplus\mathfrak{g}_{\pm\varepsilon_i}^{(1,1)}
\end{equation*}
where 
\begin{align*}
	\mathfrak{g}_{\pm\varepsilon_i}^{(1,0)}&=\mathbb{C}X_{\pm\varepsilon_i}^{(1,0)},\\
		\mathfrak{g}_{\pm\varepsilon_i}^{(1,1)}&=\mathbb{C}X_{\pm\varepsilon_i}^{(1,1)}.
\end{align*}
This is due to the fact that $\mathfrak{t}$ is not self-centralizing in $\mathfrak{g}$ and the condition of Lemma \ref{oned} is not satisfied.

\subsection{A question}
If $\mathfrak{g}$ is a basic simple $\mathbb{Z}_2 \times \mathbb{Z}_2$-graded Lie algebra, we have seen that one can associate to $\mathfrak{g}$ an abstract root system in the sense of \cite[II.5]{knappgreen} (see Corollary~\ref{abstractrootsystem}). Consequently, the full machinery of abstract root systems applies (Weyl groups, positive systems, simple roots, etc.). This abstract root system uniquely determines a Dynkin diagram (see \cite[II.5]{knappgreen}). In other words, to every basic simple $\mathbb{Z}_2 \times \mathbb{Z}_2$-graded Lie algebra one can attach a Dynkin diagram.

However, non-isomorphic basic simple $\mathbb{Z}_2 \times \mathbb{Z}_2$-graded Lie algebras may correspond to the same Dynkin diagram. For example, this occurs for $\mathfrak{so}(4,2,2,2){\mathbb{C}}$ and $\mathfrak{so}(4,4,2,0){\mathbb{C}}$, which both correspond to the Dynkin diagram of type $D_5$. Therefore, in order to classify basic simple $\mathbb{Z}_2 \times \mathbb{Z}_2$-graded Lie algebras, one must introduce "enhanced" Dynkin diagrams that incorporate the grading data.

We plan to address the classification problem in future work.

\end{document}